\documentclass[12pt]{article}

\usepackage{amsmath}
\usepackage{mathtools}
\usepackage{stmaryrd}
\usepackage{amsfonts}
\usepackage[mathscr]{eucal}
\usepackage{amsthm}
\usepackage{amssymb}
\usepackage{enumerate}
\usepackage{cases}
\usepackage{multirow}
\usepackage{subfigure}
\usepackage{graphicx}
\usepackage{appendix}
\usepackage{pgf,tikz}
\usepackage{color}
\DeclareMathOperator{\sign}{sgn}
\DeclareMathOperator{\supp}{supp}

\DeclareMathOperator{\diam}{diam}
\DeclareMathOperator{\dist}{dist}
\DeclareMathOperator{\meas}{meas}
\begin{document}
\newcommand{\done}[2]{\dfrac{d {#1}}{d {#2}}}
\newcommand{\donet}[2]{\frac{d {#1}}{d {#2}}}
\newcommand{\pdone}[2]{\dfrac{\partial {#1}}{\partial {#2}}}
\newcommand{\pdonet}[2]{\frac{\partial {#1}}{\partial {#2}}}
\newcommand{\pdonetext}[2]{\partial {#1}/\partial {#2}}
\newcommand{\pdtwo}[2]{\dfrac{\partial^2 {#1}}{\partial {#2}^2}}
\newcommand{\pdtwot}[2]{\frac{\partial^2 {#1}}{\partial {#2}^2}}
\newcommand{\pdtwomix}[3]{\dfrac{\partial^2 {#1}}{\partial {#2}\partial {#3}}}
\newcommand{\pdtwomixt}[3]{\frac{\partial^2 {#1}}{\partial {#2}\partial {#3}}}
\newcommand{\bs}[1]{\mathbf{#1}}
\newcommand{\bx}{\mathbf{x}}
\newcommand{\by}{\mathbf{y}}
\newcommand{\bd}{\mathbf{d}} 
\newcommand{\bn}{\mathbf{n}} 
\newcommand{\bP}{\mathbf{P}} 
\newcommand{\bp}{\mathbf{p}} 
\newcommand{\ol}[1]{\overline{#1}}
\newcommand{\rf}[1]{(\ref{#1})}
\newcommand{\xt}{\mathbf{x},t}
\newcommand{\hs}[1]{\hspace{#1mm}}
\newcommand{\vs}[1]{\vspace{#1mm}}
\newcommand{\eps}{\varepsilon}
\newcommand{\ord}[1]{\mathcal{O}\left(#1\right)} 
\newcommand{\oord}[1]{o\left(#1\right)}
\newcommand{\Ord}[1]{\Theta\left(#1\right)}
\newcommand{\PhiF}{\Phi_{\rm freq}}
\newcommand{\real}[1]{{\rm Re}\left[#1\right]} 
\newcommand{\im}[1]{{\rm Im}\left[#1\right]}
\newcommand{\hsnorm}[1]{||#1||_{H^{s}(\bs{R})}}
\newcommand{\hnorm}[1]{||#1||_{\tilde{H}^{-1/2}((0,1))}}
\newcommand{\norm}[2]{\left\|#1\right\|_{#2}}
\newcommand{\normt}[2]{\|#1\|_{#2}}
\newcommand{\on}[1]{\Vert{#1} \Vert_{1}}
\newcommand{\tn}[1]{\Vert{#1} \Vert_{2}}
\newcommand{\ts}{\tilde{s}}
\newcommand{\tGamma}{{\tilde{\Gamma}}}
\newcommand{\darg}[1]{\left|{\rm arg}\left[ #1 \right]\right|}
\newcommand{\bnabla}{\boldsymbol{\nabla}}
\newcommand{\dive}{\boldsymbol{\nabla}\cdot}
\newcommand{\curl}{\boldsymbol{\nabla}\times}
\newcommand{\Phixy}{\Phi(\bx,\by)}
\newcommand{\PhiOxy}{\Phi_0(\bx,\by)}
\newcommand{\dxPhixy}{\pdone{\Phi}{n(\bx)}(\bx,\by)}
\newcommand{\dyPhixy}{\pdone{\Phi}{n(\by)}(\bx,\by)}
\newcommand{\dxPhiOxy}{\pdone{\Phi_0}{n(\bx)}(\bx,\by)}
\newcommand{\dyPhiOxy}{\pdone{\Phi_0}{n(\by)}(\bx,\by)}

\newcommand{\rd}{\mathrm{d}}
\newcommand{\R}{\mathbb{R}}
\newcommand{\N}{\mathbb{N}}
\newcommand{\Z}{\mathbb{Z}}
\newcommand{\C}{\mathbb{C}}
\newcommand{\K}{{\mathbb{K}}}
\newcommand{\ri}{{\mathrm{i}}}
\newcommand{\re}{{\mathrm{e}}} 

\newcommand{\cA}{\mathcal{A}}
\newcommand{\cC}{\mathcal{C}}
\newcommand{\cS}{\mathcal{S}}
\newcommand{\cD}{\mathcal{D}}
\newcommand{\cone}{{c_{j}^\pm}}
\newcommand{\ctwo}{{c_{2,j}^\pm}}
\newcommand{\cthree}{{c_{3,j}^\pm}}

\newtheorem{thm}{Theorem}[section]
\newtheorem{lem}[thm]{Lemma}
\newtheorem{defn}[thm]{Definition}
\newtheorem{prop}[thm]{Proposition}
\newtheorem{cor}[thm]{Corollary}
\newtheorem{rem}[thm]{Remark}
\newtheorem{conj}[thm]{Conjecture}
\newtheorem{ass}[thm]{Assumption}
\newcommand{\tk}{k}
\newcommand{\tbx}{\tilde{\bx}}
\newcommand{\tby}{\tilde{\by}}	%
\newcommand{\tbd}{\tilde{\bd}}
\newcommand{\txi}{\xi}
\newcommand{\bxi}{\boldsymbol{\xi}}
\newcommand{\boldeta}{\boldsymbol{\eta}}	%
\newcommand{\sD}{\mathsf{D}}
\newcommand{\sN}{\mathsf{N}}
\newcommand{\sE}{\mathsf{E}}
\newcommand{\sS}{\mathsf{S}}
\newcommand{\sT}{\mathsf{T}}
\newcommand{\sH}{\mathsf{H}}
\newcommand{\sI}{\mathsf{I}}
\renewcommand{\cA}{A}
\newcommand{\cB}{\mathcal{B}}
\newcommand{\cH}{\mathcal{H}}
\newcommand{\cI}{\mathcal{I}}
\newcommand{\cItilde}{\tilde{\mathcal{I}}}
\newcommand{\cIhat}{\hat{\mathcal{I}}}
\newcommand{\cIcheck}{\check{\mathcal{I}}}
\newcommand{\cIstar}{{\mathcal{I}^*}}
\newcommand{\cJ}{\mathcal{J}}
\newcommand{\cM}{\mathcal{M}}
\newcommand{\cT}{\mathcal{T}}
\newcommand{\scrD}{\mathscr{D}}
\newcommand{\scrS}{\mathscr{S}}
\newcommand{\scrJ}{\mathscr{J}}
\newcommand{\tH}{\tilde{H}}
\newcommand{\Hze}{H_{\rm ze}}
\newcommand{\dimH}{{\rm dim_H}}
\newcommand{\dudnjump}{\left[ \pdone{u}{\bn}\right]}
\newcommand{\dudnjumptext}{[ \pdonetext{u}{\bn}]}
\newcommand{\sumpm}[1]{\{\!\!\{#1\}\!\!\}}
\newtheorem{example}[thm]{Example}
\title{Acoustic scattering by fractal screens: mathematical formulations and wavenumber-explicit continuity and coercivity estimates}
\author{S.\ N.\ Chandler-Wilde\footnotemark[1], D.\ P.\ Hewett\footnotemark[1] \footnotemark[2]}

\renewcommand{\thefootnote}{\fnsymbol{footnote}}
\footnotetext[1]{Department of Mathematics and Statistics, University of Reading, Whiteknights PO Box 220, Reading RG6 6AX, UK. This work was supported by EPSRC grant EP/F067798/1.}

\footnotetext[2]{Current address: Mathematical Institute, University of Oxford, Radcliffe Observatory Quarter, Woodstock Road, Oxford, OX2 6GG, UK. Email: \texttt{hewett@maths.ox.ac.uk}}

\date{Originally posted 14 Jan 2014, this corrected version 19 Mar 2015}
\maketitle
\renewcommand{\thefootnote}{\arabic{footnote}}
\begin{abstract}
We consider time-harmonic acoustic scattering by planar sound-soft (Dirichlet) and sound-hard (Neumann) screens. In contrast to previous studies, in which the domain occupied by the screen is assumed to be Lipschitz or smoother, we consider screens occupying an arbitrary bounded open set in the plane. Thus our study includes cases where the closure of the domain occupied by the screen has larger planar Lebesgue measure than the screen, as can happen, for example, when the screen has a fractal boundary. We show how to formulate well-posed boundary value problems for such scattering problems, our arguments depending on results on the coercivity of the acoustic single-layer and hypersingular boundary integral operators, and on properties of Sobolev spaces on general open sets which appear to be new. Our analysis teases out the explicit wavenumber dependence of the continuity and coercivity constants of the boundary integral operators, viewed as mappings between fractional Sobolev spaces, this in part extending previous results of Ha-Duong \cite{Ha-Du:90,Ha-Du:92}. 
We also consider the complementary problem of propagation through a bounded aperture in an infinite planar screen.
\end{abstract}
\section{Introduction}
\label{sec:Intro}
This paper concerns the mathematical analysis of classical time-harmonic acoustic scattering problems modelled by the Helmholtz equation
\begin{equation} \label {eq:he}
\Delta u + k^2 u = 0,
\end{equation}
where $k>0$ is the {\em wavenumber}. 
Our focus is on scattering by a thin planar screen occupying some bounded and relatively open set $\Gamma \subset \Gamma_\infty:=\{\bx=(x_1,...,x_n)\in \R^n:x_n=0\}$ (we assume throughout 
that $n=2$ or $3$), with \eqref{eq:he} assumed to hold in the domain $D:=\R^n\setminus \overline{\Gamma}$. 
We consider both the Dirichlet (sound-soft) and Neumann (sound-hard) boundary value problems (BVPs), which we state below. The function space notation, and the precise sense in which the boundary conditions are to be understood, will be explained in \S\ref{sec:SobolevSpaces} and \S\ref{ScatProb}.
\begin{defn}[Problem $\sD$]
\label{def:SPD}
Given $g_{\sD}\in H^{1/2}(\Gamma)$, find $u\in C^2\left(D\right)\cap  W^1_{\mathrm{loc}}(D)$ such that
\begin{eqnarray}
 & \Delta u+k^2u  =  0, \quad \mbox{in }D, & \label{eqn:HE1} \\
 & u  = g_{\sD}, \quad \mbox{on }\Gamma,&\label{eqn:bc1}
\end{eqnarray}
and $u$ satisfies the Sommerfeld radiation condition at infinity. %
\end{defn}
\begin{defn}[Problem $\sN$]
\label{def:SPN}
Given $g_{\sN}\in H^{-1/2}(\Gamma)$, find $u\in C^2\left(D\right)\cap W^1_{\mathrm{loc}}(D)$ such that
\begin{eqnarray}
 & \Delta u+k^2u  =  0, \quad \mbox{in }D, & \label{eqn:HE2} \\
 & \pdone{u}{\bn}  = g_{\sN}, \quad \mbox{on }\Gamma,& \label{eqn:bc2}%
\end{eqnarray}
and $u$ satisfies the Sommerfeld radiation condition at infinity. %
\end{defn}

\begin{example}
\label{ex:scattering}
Consider the problem of scattering by $\Gamma$ of an incident 
superposition of plane waves
\begin{align}
\label{uiDef}
    u^i(\bx)&:=\sum_{j=1}^N a_j \re^{\ri  k \bx\cdot \bd_j}, \qquad \bx\in\mathbb{R}^n,%
\end{align}
where 
$N\in\N$ and 
$\bd_j\in\mathbb{R}^n$, $j=1,\ldots,N$, are unit direction vectors. 
A `sound-soft' and a `sound-hard' screen are modelled respectively by problem $\sD$ (with $g_{\sD}=-u^i|_\Gamma$) and problem $\sN$ (with $g_{\sN}=-\pdonetext{u^i}{\bn}|_\Gamma$). In both cases $u$ represents the scattered field, the total field being given by $u^i+u$.
\end{example}

Such scattering problems have been well-studied, both theoretically \cite{StWe84,Ste:87,WeSt:90,Ha-Du:90,Ha-Du:92} and in applications \cite{DLaWroPowMan:98,DavisChew08}. However, it would appear that all previous studies assume that 
$\Gamma\subset \Gamma_\infty$ 
is at least a Lipschitz relatively open set (in the sense of \cite{McLean}), and most that $\Gamma$ is substantially smoother. In these cases it is well known (see, e.g., \cite{StWe84,Ste:87,WeSt:90}) that problems $\sD$ and $\sN$ are uniquely solvable for all $g_{\sD}\in H^{1/2}(\Gamma)$ and $g_{\sN}\in H^{-1/2}(\Gamma)$. However, this unique solvability does not hold for general non-Lipschitz screens, as we shall see.

\subsection{Main results and outline of the paper}

The current paper makes two novel contributions to the study of this problem: 
\begin{enumerate}[(i)]
\item we show how problems $\sD$ and $\sN$ must be modified so that they are uniquely solvable when $\Gamma$ is an arbitrary (in particular, fractal) bounded relatively open subset of $\Gamma_\infty$; 
\item we present new wavenumber-explicit continuity and coercivity estimates on the associated boundary integral operators (the acoustic single-layer and hypersingular operators).
\end{enumerate}

We give some motivation for and further explanation of objectives (i) and (ii) in \S\ref{sec:Motivation} below. But first we provide a brief overview of the structure of the paper.

Our approach follows previous studies (e.g.\ \cite{StWe84,Ste:87,WeSt:90}) in that it is based on the classical direct integral equation method, in which Green's theorem is used to reformulate the BVPs in the propagation domain as boundary integral equations (BIEs) on the screen. Specifically, the sound-soft BVP leads to the single-layer BIE, and the sound-hard BVP to the hypersingular BIE. To determine the solvability of these BIEs one needs to study the associated boundary integral operators (BIOs) as mappings between certain fractional Sobolev spaces defined on the screen. But while these Sobolev spaces are well-studied for Lipschitz screens, their properties for general non-Lipschitz screens do not seem to have been widely studied. We therefore begin in \S\ref{sec:SobolevSpaces} by carefully documenting some of the properties of Sobolev spaces on an arbitrary open subset $\Omega$ of $\R^n$. Of particular importance to the current study are the following two issues, which can arise when $\Omega$ is non-Lipschitz: firstly, the boundary of $\Omega$ may support non-zero distributions in $H^s(\R^n)$ for $s$ in the range relevant for BIE formulations (see \S\ref{sec:Polarity}); and secondly, it is not in general possible to approximate an arbitrary $H^s(\R^n)$ distribution supported in the closure of $\Omega$ by a sequence of smooth functions compactly supported inside $\Omega$ (see \S\ref{sec:EqualityOfSpaces}). Dealing with these two issues will prove to be crucial when it comes to correctly formulating the screen scattering problems for arbitrary screens (objective (i) above), which we do in \S\ref{ScatProb}. 

While we allow our screen to have an arbitrarily rough boundary, our proofs that the modified BVPs we propose are well-posed make extensive use of the fact that the screen is planar. This planarity means that Sobolev spaces on the screen can be defined concretely for all orders of Sobolev regularity in terms of Fourier transforms (more precisely, via Bessel potentials). 
Furthermore, it allows us to write explicit representations for the relevant layer potentials and BIOs in terms of Fourier transforms. These representations, which we present in \S\ref{sec:FourierRep}, make it possible to prove that the BIOs are coercive. They also facilitate our wavenumber-explicit analysis, presented in \S\ref{sec:SkAnalysis} and \S\ref{sec:TkAnalysis}, of the associated continuity and coercivity constants. 
In this respect our wavenumber-explicit analysis builds on and generalises that carried out using similar arguments by Ha-Duong in \cite{Ha-Du:90,Ha-Du:92}\footnote{
We are grateful to M.\ Costabel for drawing references \cite{Ha-Du:90,Ha-Du:92} to our attention.}, 
although the results presented in \cite{Ha-Du:90} are valid only for complex wavenumbers, and the results for the hypersingular operator presented in \cite{Ha-Du:92}, while valid for real wavenumbers, are not as sharp in their wavenumber dependence as the results we obtain (see \S\ref{sec:TkAnalysis} for a more detailed discussion). 
We remark that while coercivity was proved for the hypersingular BIO in \cite{Ha-Du:92}, to the best of our knowledge a proof of coercivity for the single-layer BIO does not seem to have been published before; the coercivity result is stated without proof in \cite[Prop.\ 2.3]{Co:04}, with a reference to \cite{Ha-Du:90}, but in \cite{Ha-Du:90} the result is only proved for complex wavenumber, the real case being mentioned only in passing (see \cite[p.\ 502]{Ha-Du:90}). In fact, as we shall see, the proof of coercivity for the single-layer BIO for real wavenumber is actually much more straightforward than that for the hypersingular BIO.

In \S\ref{sec:NormEstimates} we collect some useful norm estimates in the space $H^s(\Gamma)$, of relevance in the numerical analysis of Galerkin boundary element methods (BEMs) based on our integral equation formulation of the sound-soft screen problem (see e.g.\ \cite{ScreenBEM} for an application of these results). Finally, in \S\ref{ScatProbHole} we present well-posed formulations of the complementary problems of scattering by unbounded sound-soft or sound-hard screens with bounded apertures. The analysis is similar to that of the screen problem, but with some small technical differences. We note that for the aperture problem the sound-soft BVP gives rise to the hypersingular BIE and the sound-hard BVP gives rise to the single-layer BIE.

We remark that some of the results on the BVP and BIE formulations in this paper were stated without proof in the conference paper \cite{HeCh:13}. Further results concerning Sobolev spaces on non-Lipschitz domains will be presented in the future publication \cite{ChaHewMoi:13}. A BIE formulation for the case where $\Gamma$ is an arbitrary relatively \emph{closed} subset of $\Gamma_\infty$ is presented in \cite{Ch:13}, along with an application to screens defined in terms of the ``Cantor dust'' fractal.

\subsection{Motivation and background}
\label{sec:Motivation}

We end this introduction 
by giving some motivation for objectives (i) and (ii) above.

Objective (i) above is motivated in part by the possible application to the design and simulation of antennas for electromagnetic wave transmission/reception whose geometry is based on fractal subsets of the plane (see e.g.\ \cite{PBaRomPouCar:98,SriRanKri:11}). A key property of fractals is the fact that they possess structure on every lengthscale - this has been exploited to create antennas which can transmit/receive efficiently over a broad range of frequencies simultaneously. Although this seems to be a mature engineering technology, as far as we are aware no analytical framework is currently available for such problems. Of course, quantitative modelling of such problems would involve a study of the full electromagnetic wave scattering problem, but the acoustic case considered in the current paper represents a first step in this direction.

Objective (ii) forms part of a wider effort in the rigorous mathematical analysis of BIE methods for high frequency acoustic scattering problems (for a recent review of this area see e.g.\ \cite{ChGrLaSp:11}). Typically (and this is the approach adopted in the current paper) one reformulates the scattering problem as an integral equation, which can be written in operator form as
\begin{equation} \label{eq:ieab}
\cA \phi = f,
\end{equation}
where $\phi$ and $f$ are complex-valued functions defined on the boundary $\Gamma$ of the scatterer. A standard and appropriate functional analysis framework
is that the solution $\phi$ is sought in some Hilbert space $V$, with $f\in V^*$, the dual space of $V$ (the space of continuous antilinear functionals), and $\cA:V\to V^*$ a bounded linear boundary integral operator.\footnote{A concrete example is the standard Brakhage-Werner formulation \cite{BraWer:65,ChGrLaSp:11} of sound-soft acoustic scattering by a bounded, Lipschitz obstacle, in which case $V=V^* = L^2(\Gamma)$, and
\begin{equation} \label{eq:ABW}
\cA = \frac{1}{2}I + D_k -\ri \eta S_k,
\end{equation}
with $S_k$ and $D_k$ the standard acoustic single- and double-layer BIOs, $I$ the identity operator, and $\eta\in \R\setminus\{0\}$ a coupling parameter.}
Equation \rf{eq:ieab} can be restated in weak (or variational) form as
\begin{align}
\label{WeakForm}
a(\phi, \varphi) = f(\varphi), \quad \textrm{for all } \varphi\in V,
\end{align}
in terms of the sesquilinear form
\begin{align*}
\label{}
a(\phi, \varphi):=(\cA \phi)(\varphi), \quad \phi,\varphi\in V.
\end{align*}
The Galerkin method for approximating \rf{WeakForm} is to seek a solution $\phi_N\in V_N\subset V$, where $V_N$ is a finite-dimensional subspace, requiring that
\begin{align}
\label{WeakFormGalerkin}
a(\phi_N, \varphi_N) = f(\varphi_N), \quad \textrm{for all } \varphi_N\in V_N.
\end{align}

The sesquilinear form $a$ is clearly bounded with continuity constant equal to $\|\cA\|_{V\to V^*} $.
We say that $a$ (and the associated bounded linear operator $\cA$) is \emph{coercive} if, for some $\gamma>0$ (called the \emph{coercivity constant}), it holds that
\begin{align*}
\label{}
|a(\phi,\phi)|\geq \gamma\|\phi\|_{V}^2, \quad \textrm{for all } \phi\in V.
\end{align*}
In this case, the Lax-Milgram lemma implies that \rf{WeakForm} (and hence \rf{eq:ieab}) has exactly one solution $\phi\in V$, and that $\|\phi\|_{V} \leq \gamma^{-1} \|f\|_{V^*}$, i.e.\ $\|A^{-1}\|_{V^*\to V} \leq \gamma^{-1}$.
Furthermore, by Cea's lemma, the existence and uniqueness of the Galerkin solution $\phi_N$ of \rf{WeakFormGalerkin} is then also guaranteed for any finite-dimensional approximation space $V_N$, and there holds the quasi-optimality estimate
\begin{align}
\label{QuasiOpt}
\|\phi -\phi_N\|_{V} \leq \frac{\|\cA\|_{V\to V^*}}{\gamma} \inf_{\varphi_N\in V_N} \| \phi- \varphi_N\|_{V}.
\end{align}

One major thrust of recent work (for a review see \cite{ChGrLaSp:11}) has been to attempt to prove wavenumber-explicit continuity and coercivity estimates for BIE formulations of scattering problems, which, by the above discussion, lead to wavenumber-explicit bounds on the condition number $\|A\|_{V\to V^*}\|A^{-1}\|_{V^*\to V}$ and the quasi-optimality constant $\gamma^{-1}\|\cA\|_{V\to V^*}$.
\footnote{Such estimates have recently been proved \cite{SpSm:11} for the operator \rf{eq:ABW} for the case where the scatterer is strictly convex and $\Gamma$ is sufficiently smooth. We also note that in \cite{SCWGS11} a new formulation for sound-soft acoustic scattering, the so-called `star-combined' formulation, has been shown to be coercive on $L^2(\Gamma)$ for all star-like Lipschitz scatterers.}
This effort is motivated by the fact that these problems are computationally challenging when the wavenumber $k>0$ (proportional to the frequency) is large, and that such wavenumber-explicit estimates are useful for answering certain key numerical analysis questions, for instance:
\begin{enumerate}[(a)]
\item Understanding the behaviour of iterative solvers (combined with matrix compression techniques such as the fast multipole method) at high frequencies, in particular understanding the dependence of iteration counts on parameters related to the wavenumber. This behaviour depends, to a crude first approximation, on the condition number of the associated matrices, which is in part related to the wavenumber dependence of the norms of the BIOs and their inverses at the continous level (\cite{BeChGrLaLi:11,SpSm:11}).
\item Understanding the accuracy of conventional BEMs (based on piecewise polynomial approxiomation spaces) at high frequencies by undertaking a rigorous numerical analysis which teases out the joint dependence of the error on the number of degrees of freedom and the wavenumber $k$. For example, is it enough to increase the degrees of freedom in proportion to $k^{d-1}$ in order to maintain accuracy, maintaining a fixed number of degrees of freedom per wavelength in each coordinate direction? See e.g.\ \cite{LohMel:11,GrLoMeSp:13} for some recent results in this area.
\item Developing, and justifying by a complete numerical analysis, novel BEMs for high frequency scattering problems based on the so-called `hybrid numerical-asymptotic' (HNA) approach, the idea of which is to use an approximation space enriched with oscillatory basis functions, carefully chosen to capture the high frequency solution behaviour. The aim is to develop algorithms for which the number of degrees of freedom $N$ required to achieve any desired accuracy be fixed or increase only very mildly as $k\to\infty$. This aim is provably achieved in certain cases, mainly 2D so far; see, e.g., \cite{DHSLMM,NonConvex} and the recent review \cite{ChGrLaSp:11}. For 2D screen and aperture problems we recently proposed in \cite{ScreenBEM} an HNA BEM which provably achieves a fixed accuracy of approximation with $N$ growing at worst like $\log^2{k}$ as $k\to\infty$, our numerical analysis using the wavenumber-explicit continuity and coercivity estimates of the current paper. 
Numerical experiments demonstrating the effectiveness of HNA approximation spaces for a 3D screen problem have been presented in \cite[\S7.6]{ChGrLaSp:11}.
\end{enumerate}

\section{Sobolev spaces}
\label{sec:SobolevSpaces}
We now define the Sobolev spaces that we will use throughout. Our analysis is mostly in the context of the Bessel potential spaces $H^s(\R^n)$ for $s\in\R$, defined in \S\ref{ss:SobolevDef} below. In line with other analyses of high frequency scattering we use a wavenumber-dependent norm on $H^s(\R^n)$ which is equivalent to the standard norm, but allows easier derivation of wavenumber-explicit estimates. Following \cite{McLean}, we will define three types of Sobolev space on non-empty open subsets $\Omega$ of $\R^n$, namely $H^s(\Omega)$, $\tilde H^s(\Omega)$, and $H^s_{\overline{\Omega}}$. We will show in \S\ref{ss:DualAnnihIntp} that the duality relation $(\tilde H^s(\Omega))^* = H^{-s}(\Omega)$ (well-known for Lipschitz $\Omega$) holds for arbitrary open $\Omega$. 
In \S\ref{sec:Polarity} we introduce the concept of nullity, an indicator of `negligibility' of a subset of $\R^n$ in terms of Sobolev regularity which, as we shall see, is intimately related to a particular variant of the concept of \emph{capacity}. 
This concept of nullity will allow us to describe in \S\ref{sec:EqualityOfSpaces} conditions under which Sobolev spaces defined on different domains coincide. 
Finally, in \S\ref{sec:3spaces} we highlight the fact that, while the spaces $\tilde H^s(\Omega)$ and $H^s_{\overline{\Omega}}$ are known to coincide when $\Omega$ is sufficiently smooth (e.g. Lipschitz), these spaces are in general distinct. Taking account of this distinction will be crucial when we formulate well-posed boundary value problems for scattering by arbitrary screens in \S\ref{ScatProb}.

\subsection{Definitions}\label{ss:SobolevDef}
Given $n\in \N$, let $\scrD(\R^n)$ denote the space of compactly supported smooth test functions on $\R^n$, and for any open set $\Omega\subset \R^n$ let
$$
\scrD(\Omega):=\{u\in\scrD(\R^n):\supp{U}\subset\Omega\} \textrm{ and } \scrD(\overline\Omega):=\{u\in C^\infty(\Omega):u=U|_\Omega \textrm{ for some }U\in\scrD(\R^n)\}.
$$
For $\Omega\subset \R^n$ let $\scrD^*(\Omega)$ denote the space of distributions on $\Omega$, which we take as the space of anti-linear continuous functionals on $\mathscr{D}(\Omega)$. With $L^1_{\rm loc}(\Omega)$ denoting the space of locally integrable functions on $\Omega$, the standard embedding $\iota_\Omega:L^1_{\rm loc}(\Omega)\to \scrD^*(\Omega)$ is given by $\iota_\Omega u(v):=\int_\Omega u \overline{v}$ for $u\in L^1_{\rm loc}(\Omega)$ and $v\in \mathscr{D}(\Omega)$.
Let $\mathscr{S}(\R^n)$ denote the Schwartz space of rapidly decaying smooth test functions on $\R^n$, and let $\mathscr{S}^*(\R^n)$ denote the dual space of tempered distributions, which we take as the space of anti-linear continuous functionals on $\mathscr{S}(\R^n)$.
Since the inclusion $\mathscr{D}(\R^n)\subset \mathscr{S}(\R^n)$ is continuous with dense image, we have $\scrS^*(\R^n)\hookrightarrow \scrD^*(\R^n)$.
For $u\in \mathscr{S}(\R^n)$ we define the Fourier transform $\hat{u}={\cal F} u\in \mathscr{S}(\R^n)$ and $\check{u}={\cal F}^{-1} u\in \mathscr{S}(\R^n)$ by %
\begin{align}
\label{FTDef}
\hat{u}(\bxi):= \frac{1}{(2\pi)^{n/2}}\int_{\R^n}\re^{-\ri \bxi\cdot \bx}u(\bx)\,\rd \bx , \;\; \bxi\in\R^n, \quad
\check{u}(\bx) := \frac{1}{(2\pi)^{n/2}}\int_{\R^n}\re^{\ri \bxi\cdot \bx}u(\bxi)\,\rd \bxi , \;\;\bx\in\R^n.
\end{align}
We define the ($k$-dependent) Bessel potential operator $\cJ_k^s$ on $\mathscr{S}(\R^n)$, for $s\in\R$ and $k>0$, by $\cJ_k^s := {\cal F}^{-1}\cM_k^s{\cal F}$, where $\cM_k^s$ is multiplication by $m_k^s(\bxi) := (k^2+|\bxi|^2)^{s/2}$. 
We extend these definitions to $\mathscr{S}^*(\R^n)$ in the usual way: for $u\in \mathscr{S}^*(\R^n)$ and $v\in \mathscr{S}(\R^n)$ we define
\begin{align}
\label{FTDistDef}
\hat{u}(v) := u(\check{v}),\quad
\check{u}(v) := u(\hat{v}),\quad \cM_k^su(v) := u(\cM_k^s v), \quad
(\cJ_k^s u)(v) := u(\cJ_k^s v),
\end{align}
and note that for $u\in \mathscr{S}^*(\R^n)$ it holds that $\widehat{\cJ_k^s u} = \cM_k^s\hat{u}$. The standard embedding $\iota:L^2(\R^n)\to \mathscr{S}^*(\R^n)$ is given by $\iota u(v) := \iota_{\R^n}u(v)=(u,v)_{L^2(\R^n)}$, for $u\in L^2(\R^n)$, $v\in\mathscr{S}(\R^n)$.

We define the Sobolev space $H^s(\R^n)\subset \mathscr{S}^*(\R^n)$ by
\begin{align*}
H^s(\R^n):=(\cJ_k^{s})^{-1}\,\iota\big(L^2(\R^n)\big)= \cJ_k^{-s}\,\,\iota\big(L^2(\R^n)\big) = \big\{u\in \mathscr{S}^*(\R^n): \cJ_k^s u \in \iota\big(L^2(\R^n)\big)\big\},
\end{align*}
which is a Hilbert space when equipped with the inner product
$\left(u,v\right)_{H^{s}_k(\R^n)}:=\left(\iota^{-1}\cJ_k^s u,\iota^{-1}\cJ_k^s v\right)_{L^2(\R^n)}$, which makes $\cJ_k^{-s}\,\iota:L^2(\R^n)\to H^s(\R^n)$ a unitary isomorphism.
If $u\in H^s(\R^n)$ then the Fourier transform $\hat u\in \scrS^*(\R^n)$ lies in $\iota_{\R^n}(L^1_{\rm loc}(\R^n))$; that is, $\hat u$ can be identified with a locally integrable function, namely $(\iota_{\R^n}^{-1}\hat u)(\bxi) = (k^2+|\bxi|^2)^{-s/2}(\iota^{-1}\cM_k^s \hat u)(\bxi)$ for $\bxi\in \R^n$. In a slight (and standard) abuse of notation we will write $\hat u(\bxi)$ in place of $(\iota_{\R^n}^{-1}\hat u)(\bxi)$.
With this convention, for $u,v\in H^s(\R^n)$, %
\begin{align} \label{eq:HsProdNorm}%
\left(u,v\right)_{H_k^{s}(\R^n)} = \int_{\R^n}(k^2+|\bxi|^2)^{s}\,\hat{u}(\bxi)\overline{\hat{v}(\bxi)}\,\rd \bxi, \quad
\norm{u}{H_k^{s}(\R^n)}^2 = \norm{\iota^{-1} \cJ_k^s u}{L^2(\R^n)}^2 = \int_{\R^n}(k^2+|\bxi|^2)^{s}|\hat{u}(\bxi)|^2\,\rd \bxi.
\end{align}
We emphasize that, for any $k>0$, $\norm{\cdot}{H_k^{s}(\R^n)}$ is equivalent to the standard norm on $H^s(\R^n)$, defined as $\norm{\cdot}{H^{s}(\R^n)}:= \norm{\cdot}{H_1^{s}(\R^n)}$. Explicitly, 
$m\norm{u}{H^{s}(\R^n)} \leq \norm{u}{H^{s}_k(\R^n)} \leq M  \norm{u}{H^{s}(\R^n)}$,
for $u \in H^s(\R^n)$, where $m:=\min\{1,k^s\}$, $M:=\max\{1,k^s\}$. 
For any $s,t\in \R$, the map $\cJ_k^{t}:H^{s}(\R^n)\to H^{s-t}(\R^n)$ is a unitary isomorphism with inverse $\cJ_k^{-t}$. 
Also, $\scrD(\R^n)$ (more precisely, $\iota(\scrD(\R^n))$) is a dense subset of $H^s(\R^n)$ for every $s\in \R$. Hence one can view $H^s(\R^n)$ either as the space of those tempered distributions $u$ whose (distributional) Fourier transform, $\hat u$, is locally integrable on $\R^n$ and satisfies $\|u\|_{H^s(\R^n)}<\infty$, or equivalently as the completion of $\scrD(\R^n)$ with respect to the norm $\|\cdot\|_{H^s(\R^n)}$.
We note for future reference, that, for any $\bx_0\in\R^n$, the Dirac delta function
\begin{equation}\label{eq:delta}
\delta_{\bx_0}\in H^{s}(\R^n)\qquad \text{if and only if}\qquad s<-n/2.
\end{equation}

Given a closed set $F\subset \R^n$, we define the Sobolev space $H^s_F$, for $s\in \R$, by
\begin{equation} \label{HsTdef}
H_F^s := \{u\in H^s(\R^n): \supp(u) \subset F\},
\end{equation}
where the support of a distribution is understood as in \cite{McLean}. Clearly $H_F^s$ is a closed subspace of $H^s(\R^n)$. 

There are a number of ways to define Sobolev spaces on $\Omega$ when $\Omega$ is a non-empty open subset of $\R^n$. 
First, we can consider the space $H^s_{\overline{\Omega}}$, defined as in \rf{HsTdef}. 
Second, we can consider the closure of $\iota(\scrD(\Omega))$ in the space $H^s(\R^n)$, denoted
\begin{equation*}
 \tilde{H}^s(\Omega):=\overline{\iota\big(\scrD(\Omega)\big)}^{H^s(\R^n)}.
\end{equation*}
By definition, $\tilde H^s(\Omega)$ is a closed subspace of $H^s(\R^n)$, and it is easy to see that $\tilde{H}^s(\Omega)\subset H^s_{\overline\Omega}$. 
Third, let
$$
H^s(\Omega):=\{u\in \scrD^*(\Omega): u=U|_\Omega \textrm{ for some }U\in H^s(\R^n)\},
$$
where $U|_\Omega$ denotes the restriction of the distribution $U$ to $\Omega$ in the sense defined in \cite{McLean}, with norm%
\begin{align*}
\|u\|_{H_k^{s}(\Omega)}:=\inf_{\substack{U\in H^s(\R^n)\\ U|_{\Omega}=u}}\normt{U}{H_k^{s}(\R^n)}.
\end{align*}
We note that the restriction operator $|_\Omega :(H^s_{\R^n\setminus\Omega})^\perp\to H^s(\Omega)$ is a unitary isomorphism (see \cite[p.\ 77]{McLean}); here $^\perp$ denotes the orthogonal complement in $H^s(\R^n)$. 
We also remark that $\scrD(\overline\Omega)$ (more precisely, $\iota_{\Omega}(\scrD(\overline\Omega))$) is a dense subset of $H^s(\Omega)$, since $\iota(\scrD(\R^n))$ is dense in $H^s(\R^n)$.

Although we will not require them for our analysis, for completeness (and to give an indication of the richness of the theory of Sobolev spaces on non-Lipschitz open sets) we mention two further ways to define Sobolev spaces on $\Omega$. First we recall the following classical subspace of $H^s(\Omega)$:
\begin{equation}\label{eq:Hs0}
H^s_0(\Omega):=
\overline{\Big(\iota\big(\scrD(\Omega)\big)\Big)\big|_\Omega}^{H^s(\Omega)}.
\end{equation}
We note that while $\tilde H^s(\Omega)$ and $H^s_0(\Omega)$ are both defined as closures of the space of the smooth functions compactly supported in $\Omega$, they have a different nature as the former is a subspace of $H^s(\R^n)\subset \scrS^*(\R^n)$ and the latter of $H^s(\Omega)\subset\scrD^*(\Omega)$. Second, for $s\geq 0$ another natural space to consider is the set of those $H^s(\R^n)$ functions which vanish almost everywhere outside $\Omega$,
$$
\mathring H^s(\Omega) := \big\{u\in H^s(\R^n): u= 0 \mbox{ a.e. in } \R^n\setminus \Omega\big\}
= \big\{u\in H^s(\R^n): \meas\big(\supp{u}\cap(\R^n\setminus \Omega)\big) = 0\big\}
$$
which, as for $\tilde H^s(\Omega)$ and $H^s_{\overline{\Omega}}$, we equip with the inner product and norm inherited from $H^s(\R^n)$. The space $\mathring H^s(\Omega)$ can obviously be mapped bijectively (using the restriction operator) onto the set of $H^s(\Omega)$ distributions whose extension by zero outside of $\Omega$ defines an element of $H^s(\R^n)$. 
A discussion of the relationship between all of these spaces, along with many other results on Sobolev spaces on general non-Lipschitz open sets, can be found in \cite{ChaHewMoi:13}. 

Sobolev spaces can also be defined, for $s\geq 0$, as subspaces of $L^2(\R^d)$ satisfying constraints on weak derivatives. In particular, given a non-empty open subset $\Omega$ of $\R^d$, let
$$
W^1(\Omega) := \{u\in L^2(\Omega): \nabla u \in L^2(\Omega)\},
$$
where $\nabla u$ is the weak gradient. Note that $W^1(\R^d)=H^1(\R^d)$ with
$$
\|u\|^2_{H^1_k(\R^d)} = \int_{\R^d} \left( |\nabla u(\bx)|^2 + k^2 |u(\bx)|^2\right) \rd \bx.
$$
 Further \cite[Theorem 3.30]{McLean}, $W^1(\Omega)= H^1(\Omega)$ whenever $\Omega$ is a Lipschitz open set, in the sense of, e.g., \cite{sauter-schwab11, ChGrLaSp:11}.   It is convenient to use the notation
$$
W^1_{\mathrm{loc}}(\Omega) := \{u\in L^2_{\mathrm{loc}}(\Omega): \nabla u \in L^2_{\mathrm{loc}}(\Omega)\},
$$
where $L^2_{\mathrm{loc}}(\Omega)$ denotes the set of locally integrable functions $u$ on $\Omega$ for which $\int_G|u(\bx)|^2 \rd \bx < \infty$ for every bounded measurable $G\subset \Omega$.
\subsection{Duality relations}\label{ss:DualAnnihIntp}
It is standard (see e.g.\ \cite{McLean}) that $H^{-s}(\R^n)$ is a natural realisation of $(H^s(\R^n))^*$, the dual space of bounded antilinear functionals on $H^s(\R^n)$. Explicitly, where $R:H^s(\R^n)\to (H^s(\R^n))^*$ is the Riesz isomorphism, the map  $\cI:=R \cJ^{2s}_k:H^{-s}(\R^n)\to (H^s(\R^n))^*$ is a unitary isomorphism. We have
$$
\cI u(v) = \left\langle u, v \right\rangle_{s}, \quad u\in H^{-s}(\R^n), \; v\in H^s(\R^n),
$$
 where $\left\langle \cdot, \cdot \right\rangle_{s}$ is the standard Sobolev space duality pairing, the continuous sesquilinear form on $H^{-s}(\R^n)\times H^{s}(\R^n)$ defined by %
\begin{align}
\label{DualDef}
\left\langle u, v \right\rangle_{s}:=\left(\iota^{-1}\cJ_k^{-s} u,\iota^{-1}\cJ_k^{s} v\right)_{L^2(\R^n)} = \int_{\R^n}\hat{u}(\bxi) \overline{\hat{v}(\bxi)}\,\rd \bxi.
\end{align}
This unitary realisation of the dual space is attractive because the associated duality pairing \rf{DualDef} is simply the $L^2(\R^n)$ inner product when $u,v\in \iota(\scrS(\R^n))$, and a continuous extension of that inner product for $u\in H^{-s}(\R^n)$, $v\in H^{s}(\R^n)$. Moreover, if $u\in H^{-s}(\R^n)$ and $v\in \iota(\scrS(\R^n))\subset H^s(\R^n)$, then, using \eqref{FTDistDef} and \eqref{DualDef},
\begin{equation} \label{dualequiv}
\langle u, v\rangle_s  = (\iota^{-1}\cJ_k^{-s}u, \iota^{-1}\cJ_k^{s}v)_{L^2(\R^n)} = \cJ_k^{-s}u(\iota^{-1} \cJ_k^{s}v) = \cJ_k^{-s}u(\cJ_k^{s}\iota^{-1}v) = u(\iota^{-1}v),
\end{equation}
so that the duality pairing $\langle u, v\rangle_s$ is simply the action of the tempered distribution $u$ on $\iota^{-1}v\in \scrS(\R^n)$.

For ease of presentation, in the rest of the paper we will omit to write the operators $\iota$ and $\iota_\Omega$, thus identifying functions and distributions in the usual way.

Comparing \rf{eq:HsProdNorm} and \rf{DualDef} we see that $\mathcal{J}_k^{2s}:H^{s}(\R^n)\to H^{-s}(\R^n)$ satisfies
\begin{align}
\label{DualvsIP}
\langle \mathcal{J}_k^{2s}u,v \rangle_{s} = (u,v)_{H^s_k(\R^n)}, 
\qquad u,v\in H^{s}(\R^n).
\end{align}

For $s\in \R$ there are natural embeddings 
$\cI_s:H^{-s}(\Omega)\to(\tilde{H}^s(\Omega))^*$ 
and 
$\cI_s^*:\tilde{H}^{s}(\Omega)\to(H^{-s}(\Omega))^*$ 
given by (cf.\ e.g.\ \cite[Theorem 3.14]{McLean})
\begin{align}
\label{def_embed1App}
(\cI_s u)(v)
& := \langle u,v \rangle_{H^{-s}(\Omega)\times \tilde{H}^s(\Omega)}:=\langle U,v \rangle_{s},\\
\label{def_embed2App}
(\cI_s^*v)(u)
&:= \langle v,u \rangle_{\tilde H^{s}(\Omega)\times H^{-s}(\Omega)}:=\langle v,U \rangle_{-s},
\end{align}
where $U\in H^{-s}(\R^n)$ is any extension of $u$ with $U|_\Omega=u$. 
Indeed, $\cI_s^*$ can be viewed as the Banach space adjoint (or transpose) of $\cI_s$ in the sense, e.g., of Kato \cite{Ka:95}. In particular, note that
$$
\langle v,u \rangle_{\tilde H^{s}(\Omega)\times H^{-s}(\Omega)} = \overline{\langle u,v \rangle}_{H^{-s}(\Omega)\times \tilde{H}^s(\Omega)}, \quad v\in \tilde H^{s}(\Omega), \; u\in H^{-s}(\Omega).
$$
In fact, as Theorem \ref{DualityTheorem} below states, the natural embeddings $\cI_s$ and $\cI_s^*$ %
are unitary isomorphisms, and in this sense it holds that
\begin{align}
\label{isdual}
H^{-s}(\Omega)\cong(\tilde{H}^s(\Omega))^* \; \mbox{ and }\; \tilde H^{s}(\Omega)\cong(H^{-s}(\Omega))^*.
\end{align}
We remark that the representations \eqref{isdual} for the dual spaces are well known when $\Omega$ is sufficiently regular. However, it is not widely  appreciated, at least in the numerical PDEs community, that \eqref{isdual} holds without any constraint on the geometry of $\Omega$. For example, $H^s(\Omega)$ and $\tilde H^s(\Omega)$ are defined precisely as above for an arbitrary open set $\Omega$ in \cite{Steinbach}, but \eqref{isdual} is claimed there only for the case $\Omega$ Lipschitz. In \cite{McLean} \eqref{isdual} is shown under the less restrictive condition that $\Omega$ is $C^0$, but not claimed for an arbitrary open set $\Omega$.

\begin{thm}\label{DualityTheorem}
Let $\Omega$ be any non-empty open subset of $\R^n$, and $s\in\R$. Then the mappings $\cI_s:H^{-s}(\Omega)\to(\tilde{H}^s(\Omega))^*$ and $\cI_s^*:\tilde{H}^{s}(\Omega)\to(H^{-s}(\Omega))^*$ defined by \rf{def_embed1App} and \rf{def_embed2App} are unitary isomorphisms.
\end{thm}

\begin{proof}
Noting that for Hilbert spaces the concepts of unitarity and isometricity are equivalent, it suffices to show that $\cI_s$ and $\cI_s^*$ are isometric isomorphisms. We give the proof only for $\cI_s$; the result for $\cI_s^*$ follows by taking adjoints. 
To check that $\cI_s u$ is well-defined as an element of $(\tilde{H}^s(\Omega))^*$, we first note that the right-hand side of \rf{def_embed1App} is independent of the choice of extension $U$. To check this, it is sufficient to note that if $U,U'\in H^{-s}(\R^n)$ satisfy $U|_\Omega = U'|_\Omega$ then $\langle U-U',v\rangle_{s}=0$ for all $v\in \scrD(\Omega)$, and that this result extends to all $v\in \tilde{H}^s(\Omega)$ by density. The antilinearity of $\cI_s u$ is clear, and it is easy to check that the map $\cI_s$ is bounded, with $\norm{\cI_s}{}\leq 1$.

The map $\cI_s$ is injective since if $\cI_s u=0$ then for any extension $U\in H^s(\R^n)$ of $u$ and any $v\in \scrD(\Omega)$ we have $\langle U,v\rangle_{s}=0$, which implies that $u=U|_\Omega=0$.
To show that $\cI_s$ is surjective, we first note that the map
$\mathcal{Q}:\tilde{H}^{s}(\Omega)\to H^{-s}(\Omega)$ defined by $\mathcal{Q}u:=(\mathcal{J}_k^{2s}u)|_\Omega$ is bounded with $\norm{\mathcal{Q}}{}\leq 1$. Furthermore, the map $\mathcal{R}:=\cI_s\circ\mathcal{Q}:\tilde{H}^s(\Omega)\to (\tilde{H}^s(\Omega))^*$ satisfies, by \rf{DualvsIP},
\begin{align*}
\label{}
(\mathcal{R}u)(v)
= \langle \mathcal{J}_k^{2s}u,v \rangle_{s}
= (u,v)_{H^s_k(\R^n)}
= (u,v)_{\tilde{H}^s_k(\Omega)},
\qquad \qquad u,v\in \tilde{H}^s(\Omega).
\end{align*}
But, by the Riesz representation theorem, $\mathcal{R}$ is an isometric isomorphism. This implies that $\cI_s$ is surjective, and hence bijective. Moreover, since $\cI_s^{-1}=\mathcal{Q}\circ \mathcal{R}^{-1}$ and $\normt{\mathcal{Q}}{}\leq 1$, it follows that $\normt{\cI_s^{-1}}{}\leq 1$, so that in fact $\normt{\cI_s}{}=\normt{\cI_s^{-1}}{}=1$, i.e.~$\cI_s$ is an isometry.
\end{proof}

Similarly, one can show that there are natural unitary isomorphisms which make
\begin{align}
\label{isdual3}
(\tilde{H}^{-s}(\R^n\setminus F))^\perp \cong(H^s_F)^* \; \mbox{ and }\; H^s_F\cong((\tilde{H}^{-s}(\R^n\setminus F))^\perp)^*;
\end{align}
for a more detailed discussion see \cite{ChaHewMoi:13}. Of course, \rf{isdual} and \rf{isdual3} are related by the fact that \mbox{$H^{-s}(\Omega)\cong (H^{-s}_{\R^n\setminus \Omega})^\perp$} (cf.\ \S\ref{ss:SobolevDef}).

\subsection{Nullity and capacity}
\label{sec:Polarity}
In order to compare Sobolev spaces defined on different open sets (which we do in \S\ref{sec:DifferentDomains}), and to study the relationship between the spaces $\tilde H^s(\Omega)$ and $H^s_{\overline{\Omega}}$ on a given open set $\Omega$ (which we do in \S\ref{sec:3spaces}), we require the concept of $s$-nullity of subsets of $\R^n$, which can be thought of as an indicator of negligibility in the sense of Sobolev regularity. 
\begin{defn}
For $s\in\R$ we say that a set $E\subset\R^n$ is $s$-null if there are no non-zero elements of $H^{s}(\R^n)$ supported entirely inside $E$ (equivalently, if $H^{s}_{E'}=\{0\}$ for every closed set $E'\subset E$).
\end{defn}
\begin{rem}
While the term ``$s$-null'' appears to be new, the concept it describes is very natural and has been considered elsewhere in different contexts (see, e.g., the book by Maz'ya \cite{Maz'ya}). For integer $s<0$ our definition of $s$-nullity coincides with the notion of ``$(2,-s)$-polarity'' defined in \cite[\S 13.2]{Maz'ya}. For integer $s>0$, our notion of $s$-nullity is related to the concept of ``sets of uniqueness'' (cf.\ \cite[p692]{Maz'ya}). 
The reason Maz'ya uses two different terminologies for the positive and negative order spaces is not explained in \cite{Maz'ya}, but we expect this is due to the fact that Maz'ya works primarily with the Sobolev spaces $W^s$, where the positive order spaces are defined using weak derivatives, and the negative order spaces are defined by duality. By contrast, in the Bessel potential framework of the current paper, the spaces $H^s$ are defined in the same way for all $s\in\R$; hence it seems natural to define the notion of ``negligibility'' in the same way for all $s\in\R$. Our choice of ``$s$-nullity'' as the terminology for this concept over Maz'ya's teminology ``$(-s)$-polarity'' was made simply to simplify the presentation of the results which follow and make the arguments easier to read (we find it more natural to say that a set which does not support an $H^{s}(\R^n)$ distribution is ``$s$-null'' rather than ``$(-s)$-polar''). But the difference is essentially semantic, so readers familiar with the concept of polarity should read ``$(-s)$-polar'' for ``$s$-null'' throughout.
\end{rem}

The following lemma collects a number of basic facts about $s$-nullity, relating the concept to topological and geometrical properties of a set. In particular, the results in parts
\rf{gg} and \rf{hh} %
provide a partial characterization of the $s$-null sets for $-n/2\leq s<0$ in terms of Hausdorff dimension, (defined e.g.\ in \cite[\S 5.1]{AdHe}), which we denote here, for an arbitrary set $E\subset R^n$, by ${{\rm dim_H}}(E)$.
We remark that the results in \cite[Chapter 5]{AdHe} in fact allow a slightly more precise characterization of $s$-null sets for $-n/2\leq s<0$ in terms of Hausdorff measures, but the results in parts \rf{gg} and \rf{hh} %
seem sufficient for the applications of scattering by fractal screens that motivate the current study (and moreover we note that Adams and Hedberg's remark \cite[\S5.6.4]{AdHe} implies that no \emph{complete} characterization in terms of Hausdorff measure is possible).

\begin{lem}
\label{lem:polarity}
Let $E$ be any subset of $\R^n$, and $s\in\R$. Then:
\begin{enumerate}[(i)]
\item \label{aa}If $E$ is $s$-null then any subset $E'$ of $E$ is also $s$-null.
\item \label{bb}If $E$ is $s$-null then $E$ is also $t$-null for all $t>s$.
\item \label{cc}If $E$ is $s$-null then $E$ must have empty interior.
\item \label{dd}If $s>n/2$ then $E$ is $s$-null if and only if $E$ has empty interior.
\item \label{ee}If $s\geq0$ then any set $E$ with zero Lebesgue measure is $s$-null.
\item \label{ff}$E$ is $0$-null if and only if $E$ has zero Lebesgue measure.
\item \label{gg} For $-n/2\leq s<0$, if $E$ is non-empty, Borel %
and $s$-null, then ${\rm dim_H}(E)\leq n+2s$.
\item \label{hh} For $-n/2<s<0$, if $E$ is non-empty and Borel %
and ${\rm dim_H}(E)< n+2s$, then $E$ is $s$-null.
\item \label{ii} If $s=-n/2$, any finite set is $s$-null.
\item \label{jj}If $s<-n/2$ then there are no non-empty $s$-null sets.
\item \label{kk0} If $\Omega$ is a $C^0$ open set then $\partial\Omega$ is $s$-null if $s\geq 0$.
\item \label{kk1} If $\Omega$ is a $C^0$ open set then $\partial\Omega$ is not $s$-null if $s<-1/2$.
\item \label{kk2} If $\Omega$ is a $C^{0,\alpha}$ open set for some $0<\alpha<1$ then $\partial\Omega$ is $s$-null if $s> -\alpha/2$.
\item \label{kk} If $\Omega$ is a Lipschitz open set then $\partial\Omega$ is $s$-null if and only if $s\geq -1/2$.
\item \label{ll} Let $F_1$ and $F_2$ be closed, $s$-null subsets of $\R^n$. Then $F_1\cup F_2$ is $s$-null.
\end{enumerate}
\end{lem}
\begin{proof}
\rf{aa} and \rf{bb} These follow straight from the definition of $s$-nullity. \rf{cc} If $E$ has non-empty interior one can trivially construct a non-zero element of $C^\infty_0(\R^n)\subset H^{s}(\R^n)$ supported inside $E$. \rf{dd} In this case $H^{s}(\R^n)$ consists of continuous functions (by the Sobolev embedding theorem \cite[Theorem 3.26]{McLean}). \rf{ee} and \rf{ff} follows from the fact that $H^{s}(\R^n)$ is continuously embedded inside $L^2(\R^n)$ for $s\geq 0$. \rf{gg} and \rf{hh} are proved below. \rf{ii} If $E$ is a finite set, then any distribution supported on $E$ is necessarily a linear combination of delta functions and their derivatives supported on the points of $E$ \cite[Theorem 3.9]{McLean}, but delta functions are not contained in $H^{-n/2}(\R^n)$ (see \eqref{eq:delta}).
\rf{jj} In this case $H^{s}(\R^n)$ does contain delta functions, so any non-empty set $E$ supports non-zero elements of $H^{s}(\R^n)$.
\rf{kk0} follows from \rf{ee} and the fact that the graph of a continuous function has zero Lebesgue measure (which can be seen by considering the measure of the union of infinitely many vertical translates of its graph).
\rf{kk1} The case $n=1$ is covered by \rf{jj}, so assume that $n\geq 2$. Suppose, for a contradiction, that $\partial\Omega$ is $s$-null for some $s<-1/2$. Without loss of generality we can, by \rf{bb}, assume that $-n/2\leq s <-1/2$. Then, by \rf{gg}, $\dimH{\partial\Omega}\leq n+2s<n-1$.
But this contradicts 
the fact 
that $\dimH{\partial\Omega}\geq n-1$. Hence no such $s$ can exist.
\rf{kk2} follows from \rf{hh} and 
the fact that $\dimH{\partial\Omega}\leq n-\alpha$. 
\rf{kk} is proved at the end of this section.
\rf{ll}
We first notice the following fact:
if $U\in \scrD^*(\R^n)$ and $\phi\in \scrD(\R^n)$,
and if there exists $\bx\in\supp{U}$ such that $\phi(\bx)\neq 0$, then $\phi U\neq 0$ as a distribution on $\R^n$.
To prove this, suppose that there exists $\bx\in\supp{U}$ such that $\phi(\bx)\neq 0$. Then for $\eps>0$ let $B_\eps(\bx)$ be the ball of radius $\eps$ centred on $\bx$.
Choose $\eps$ such that $\phi$ is non-zero in $B_\eps(\bx)$.
Then, since $\bx\in\supp{U}$, $U|_{B_\eps(\bx)}\neq 0$ and so $U(\psi)\neq0$ for some $\psi\in \scrD(B_\eps(\bx))$.
But then, defining $\varphi\in \scrD(\R^n)$ by $\varphi(\bx):=\psi/\phi$, for $\bx\in B_\eps(\bx)$, and $\varphi(\bx):=0$ otherwise, we have $(\phi U)(\varphi) = U(\psi)\neq 0$, so $\phi U$ is non-zero.
To prove \rf{ll} we argue again by contrapositive. Suppose that $F_1\cup F_2$ is not $s$-null, i.e.\ there exists $0\neq u\in H^{s}(\R^n)$ with $\supp{u}\subset F_1\cup F_2$. Then if $\supp{u}\subset F_1$ or $\supp{u}\subset F_2$ we are done. Otherwise, suppose w.l.o.g.\ that there exists $\bx\in\supp{u}\cap(F_1\setminus F_2)$. Then since $F_2$ is closed, $\eps:=\dist(\bx,F_2)>0$. Let $\phi\in\scrD(B_\eps(\bx))$ with $\phi(\bx)\neq 0$.
Then by the result mentioned above $0\neq\phi u\in H^{s}(\R^n)$ with $\supp{\phi u}\subset F_1$, so that $F_1$ is not $s$-null.
\end{proof}

\begin{rem}
\label{rem:TriebelHausdorff}
Parts \rf{gg} and \rf{hh} of Lemma \ref{lem:polarity} imply that if $E$ is Borel and has zero Lebesgue measure, then
$$\dim_H(E) = \inf\big\{d: E \mbox{ is } \big((d-n)/2\big)\mbox{-null}\big\},$$
which is similar to \cite[Theorem 17.8]{Triebel97FracSpec} in the special case $p=q=2$.
\end{rem}

\begin{rem}
Part \rf{ll} of Lemma \ref{lem:polarity} is not true for general subsets.
A simple counterexample is where $F_1$ comprises the rational elements of $[0,1]$ and $F_2$ the irrational elements.
Then for $s>1/2$ both $F_1$ and $F_2$ are $s$-null, since they both have empty interior.
But $F_1\cup F_2=[0,1]$, which is not $s$-null for any $s$ (since it has non-empty interior).
\end{rem}

The remainder of this section is dedicated to the proof of parts \rf{gg}, \rf{hh} and \rf{kk} of Lemma \ref{lem:polarity}.
For this we require an alternative characterisation of $s$-null sets in terms of a set function called \emph{capacity}.
Since the notion of capacity is not used elsewhere in the paper, readers wishing to skip the remainder of this section can safely proceed to \S\ref{sec:3spaces}.

\begin{defn}
\label{def:cap}
For $s\in\R$ and $K\subset\R^n$ compact, define
(cf.\ \cite[\S 13.1]{Maz'ya}, where ${\rm Cap_s}(\cdot)$ is denoted ${\rm Cap}_s(\cdot,H^s(\R^n))$)
\begin{align*}
\label{}
{\rm Cap}_s(K):= \inf  \{\|u \|^2_{H^s(\R^n)}:u\in \scrD(\R^n) \textrm{ and } u=1 \textrm{ in a neighbourhood of } K \}.
\end{align*}
For arbitrary $E\subset\R^n$ define \emph{inner} and \emph{outer} capacities by
\begin{align*}
{\rm \underline{Cap}}{}_{s}(E)&:=\sup_{\substack{K\subset E\\ K \textrm{ compact}}} {\rm Cap}_s(K),
\qquad\qquad %
{\rm \overline{Cap}}_{s}(E):=\inf_{\substack{G\supset E\\ G \textrm{ open}}} {\rm \underline{Cap}}{}_{s}(G).
\end{align*}
Clearly ${\rm \underline{Cap}}{}_{s}(E)\leq{\rm \overline{Cap}}_{s}(E)$ for all $E\subset\R^n$. If ${\rm \underline{Cap}}{}_{s}(E)={\rm \overline{Cap}}_{s}(E)$ then we say $E$ is $s$-\emph{capacitable} and define the capacity of $E$ to be ${\rm Cap}_{s}(E):={\rm \underline{Cap}}{}_{s}(E)={\rm \overline{Cap}}_{s}(E)$ (cf.\ \cite[p.~28]{AdHe}, where ${\rm \overline{Cap_s}}$ is denoted $N_{s,2}$).
We note that for $s\geq 0$ all Borel subsets of $R^n$ (in particular all compact subsets) are $s$-capacitable \cite[Theorem 2.3.11]{AdHe}).
\end{defn}
Note that we have defined ${\rm Cap_s}(\cdot)$ for all $s\in \R$ but we only actually use it for $s>0$.

The following equivalence is stated and proved by Maz'ya in \cite[\S 13.2]{Maz'ya} for integer $s<0$. In fact the proof given in \cite{Maz'ya} works for all $s\in \R$ (note that we only actually use this result for $s<0$).
\begin{lem}
\label{lem:polar_cap_equiv}
Let $s\in\R$. Then a non-empty set $E\subset\R^n$ is $s$-null if and only if $\underline{\rm Cap}{}_{-s}(E)=0$.
\end{lem}

Adams and Hedberg \cite{AdHe} also describe properties of the capacity ${\rm Cap}_s$, working mostly with the outer capacity ${\rm \overline{Cap}}_{s}(E)$, which they denote $N_{s,2}$ (cf.\ \cite[\S 5.7]{AdHe}). The results in \cite{AdHe} provide a partial characterization of the sets of zero outer capacity ${\rm \overline{Cap}}_{s}(E)$ for $0<s\leq n/2$ in terms of Hausdorff dimension,
which, when combined with Lemma \ref{lem:polar_cap_equiv}, proves parts \rf{gg} and \rf{hh} of Lemma \ref{lem:polarity}.
\begin{thm}\label{th:CapDimH}
Let $E\subset\R^n$. Then:
\begin{enumerate}[(i)]
\item For $0<s\leq n/2$, if ${\rm \overline{Cap}}_{s}(E)=0$ then ${\rm dim_H}(E)\leq n-2s$.
\item For $0<s<n/2$, if ${\rm dim_H}(E)< n-2s$ then ${\rm \overline{Cap}}_{s}(E)=0$.
\end{enumerate}
\end{thm}
\begin{proof}
This follows from Corollary 3.3.4, Theorem 5.1.9 and Theorem 5.1.13 of \cite{AdHe}.
\end{proof}

\begin{rem}
As a consequence of Theorem \ref{th:CapDimH} and Lemma \ref{lem:polar_cap_equiv}, we note that part \rf{gg} of Lemma \ref{lem:polarity} holds more generally for a $(-s)$-capacitable set $E$; part \rf{hh} holds for any set $E\subset\R^n$.
\end{rem}

Lemma \ref{lem:polar_cap_equiv} also allows us to prove part \rf{kk} of Lemma \ref{lem:polarity}.

\begin{proof}[Proof of Lemma \ref{lem:polarity}\rf{kk}]
By a standard partition of unity argument, it suffices to consider the case of a Lipschitz hypograph. Using Lemma \ref{lem:polar_cap_equiv} and the fact that the image of a set of zero capacity under a Lipschitz map has zero capacity \cite[Theorem 5.2.1]{AdHe}, we deduce that if $\Omega$ is a Lipschitz hypograph and $-n/2\leq s<0$, $\partial\Omega$ is $s$-null if and only if the hyperplane $\{\bx\in\R^n:x_n=0\}$ is $s$-null. But, by \cite[Lemma 3.39]{McLean}, a hyperplane is $s$-null if and only if $s\geq -1/2$. So for $n\geq 2$ we conclude that $\partial\Omega$ is $s$-null if and only if $s\geq -1/2$; for $n=1$ we lack the ``only if'' statement, but this follows because there are no non-empty $s$-null sets for $s<-n/2=-1/2$ (cf.\ Lemma \ref{lem:polarity}\rf{jj}).
\end{proof}
\subsection{Equality of spaces defined on different domains}
\label{sec:EqualityOfSpaces}
Using the concept of $s$-nullity defined in the previous section we can give a characterisation of when Sobolev spaces defined on different open or closed sets are in fact equal. For two subsets $E_1$ and $E_2$ of $\R^n$ we use the notation $E_1\ominus E_2$ to denote the symmetric difference between $E_1$ and $E_2$, i.e.
\begin{align*}
\label{}
E_1\ominus E_2:=(E_1\setminus E_2)\cup( E_2\setminus E_1)=(E_1\cup E_2)\setminus (E_1\cap E_2).
\end{align*}

\label{sec:DifferentDomains}
\begin{thm}
\label{thm:Hs_equality_closed}
Let $F_1$, $F_2$ be closed subsets of $\R^n$, and let $s\in\R$. Then the following are equivalent:
\begin{enumerate}[(i)]
\item \label{a}$F_1\ominus F_2$ %
is $s$-null.
\item \label{b} $F_1\setminus F_2$ and $ F_2\setminus F_1$ are both $s$-null.
\item \label{c}$H^s_{F_1\cap F_2}=H^s_{F_1}=H^s_{F_2}=H^s_{F_1\cup F_2}$.
\end{enumerate}
\end{thm}
\begin{proof}
That \rf{a} $\Rightarrow$ \rf{b} follows from Lemma \ref{lem:polarity}\rf{aa}. To show that \rf{c} $\rightarrow$ \rf{a} we argue by contrapositive: if $F_1\ominus F_2$ is not $s$-null, then $H^s_{F_1\cap F_2}\neq H^s_{F_1\cup F_2}$. It remains to prove that \rf{b} $\rightarrow$ \rf{c}. For this it suffices to prove that, if $A,B\subset \R^n$ are closed and $A\subset B$, then if $B\setminus A$ is $s$-null, it holds that $H^s_A=H^s_B$. To see this we argue by contrapositive.
Suppose that $H^s_A\neq H^s_B$. Then there exists $0\neq u\in H^s(\R^n)$ with $\supp{u}\cap(B\setminus A)$ non-empty. Let $\bx\in\supp{u}\cap(B\setminus A)$, and (by the closedness of $A$), let $\eps>0$ be such that $B_\eps(\bx)\cap A$ is empty. Then, for any $\phi\in \scrD(B_\eps(\bx))$ such that $\phi(\bx)\neq 0$, it holds (cf.\ the proof of Lemma \ref{lem:polarity}\rf{ll}) that $0\neq \phi u \in H^s(\R^n)$ with $\supp(\phi u)\subset B\setminus A$, which implies that $B\setminus A$ is not $s$-null.
\end{proof}

From Theorem \ref{thm:Hs_equality_closed} one can deduce a corresponding result about spaces defined on open subsets.
\begin{thm}
\label{thm:Hs_equality_open}
Let $\Omega_1$, $\Omega_2$ be non-empty, open subsets of $\R^n$, and let $s\in\R$. Then the following are equivalent:
\begin{enumerate}[(i)]
\item \label{a0}$\Omega_1\ominus\Omega_2$ %
is $s$-null.
\item \label{b0}$\Omega_1\setminus\Omega_2$ and $\Omega_2\setminus\Omega_1$ are both $s$-null.
\item \label{c0}$\Omega_1\cap\Omega_2$ is non-empty and $H^{s}(\Omega_1\cap \Omega_2)=H^{s}(\Omega_1)=H^{s}(\Omega_2)=H^{s}(\Omega_1\cup \Omega_2)$, in the sense that
$\big(H^s_{\R^n\setminus (\Omega_1\cap \Omega_2)}\big)^\perp
=\big( H^s_{\R^n\setminus \Omega_1}\big)^\perp
=\big( H^s_{\R^n\setminus \Omega_2}\big)^\perp
=\big( H^s_{\R^n\setminus (\Omega_1\cup \Omega_2)}\big)^\perp$
(recall the identification $H^s(\Omega)\cong (H^s_{\R^n\setminus\Omega})^\perp$ discussed in \S\ref{ss:SobolevDef}).
\item \label{d0}$\Omega_1\cap\Omega_2$ is non-empty and $\tilde H^{-s}(\Omega_1\cap \Omega_2)=\tilde H^{-s}(\Omega_1)=\tilde H^{-s}(\Omega_2)=\tilde H^{-s}(\Omega_1\cap \Omega_2)$.
\end{enumerate}
\end{thm}
\begin{proof}
The result follows from applying Theorem \ref{thm:Hs_equality_closed} with $F_j:=\R^n\setminus \Omega_j$, $j=1,2$, and from the duality Theorem \ref{DualityTheorem}.
\end{proof}

\begin{rem}
We note that for non-empty, open sets $\Omega_1$, $\Omega_2$, the symmetric difference $\Omega_1\ominus\Omega_2$ has empty interior if and only if 
\begin{align}
\label{eqn:ClosureEquality}
\overline{\Omega_1\cap \Omega_2} = \overline{\Omega_1} =\overline{\Omega_2}= \overline{\Omega_1\cup \Omega_2}. 
\end{align}
Thus, by Lemma \ref{lem:polarity}\rf{cc},\rf{dd}, \rf{eqn:ClosureEquality} is a necessary condition for the statements \rf{a0}--\rf{d0} of Theorem \ref{thm:Hs_equality_open} to hold; \rf{eqn:ClosureEquality} is also a sufficient condition when $s>n/2$, but
not in general for smaller $s$.
\end{rem}

\subsection{The relationship between $\tilde H^s(\Omega)$ and $H^s_{\overline{\Omega}}$}
\label{sec:3spaces}
For a non-empty open subset $\Omega$ of $\R^n$ and for $s\in\R$, the spaces $\tilde H^s(\Omega)$ and $H^s_{\overline{\Omega}}$ are closed subspaces of $H^s(\R^n)$ which satisfy the inclusion $\tilde{H}^s(\Omega)\subset H^s_{\overline\Omega}$. Under certain assumptions on $\Omega$ these two spaces are in fact equal.
\begin{lem}[{\cite[Theorem 3.29]{McLean}}]
 \label{lem:sob_equiv}
Let $\Omega\subset \R^n$ be non-empty, open, and $C^0$ (in the sense of \cite[p.\ 90]{McLean}), and let $s\in\R$.
Then $\tilde{H}^s(\Omega)=H^s_{\overline\Omega}$.
\end{lem}
However, this equality does not hold for general open subsets $\Omega$. In particular we note the following general results, which, despite their simplicity, appear to be new (for a more detailed discussion see \cite{ChaHewMoi:13}). Here, and in the sequel, $m$ denotes the $n$-dimensional Lebesgue measure on $\R^n$. 
\begin{lem}
\label{lem:ss_neq}
Let $\Omega\subset \R^n$ be non-empty and open. 
\begin{enumerate}[(i)]
\item \label{aa0} $\tilde{H}^0(\Omega)=H^0_{\overline{\Omega}}$ if and only if $m(\partial\Omega)=0$;
\item \label{bb0} If $0<s\leq n/2$ and $\dimH(\mathrm{int}(\overline{\Omega})\setminus \Omega)>n-2s$ then $\tilde H^s(\Omega) \subsetneqq H^s_{\overline{\Omega}}$;
\item \label{cc0} If $s>n/2$ and $\Omega\neq \mathrm{int}(\overline{\Omega})$ then $\tilde H^s(\Omega) \subsetneqq H^s_{\overline{\Omega}}$.
\end{enumerate}
\end{lem}
\begin{proof}
\rf{aa0} holds because $\tilde H^0(\Omega) = L^2(\Omega)$ and $H^0_{\overline{\Omega}}=L^2(\overline{\Omega})$. For \rf{bb0} and \rf{cc0}, if either $0<s\leq n/2$ and $\dimH(\mathrm{int}(\overline{\Omega})\setminus \Omega)>n-2s$ or $s>n/2$ and $\Omega\neq \mathrm{int}(\overline{\Omega})$, then (by Lemma \ref{lem:polarity}\rf{gg} or \rf{jj} respectively) $\mathrm{int}(\overline{\Omega})$ is $(-s)$-null, and hence $\tilde H^s(\Omega) \subsetneqq \tilde H^s(\mathrm{int}(\overline{\Omega})) \subset H^s_{\overline{\Omega}}$ by Theorem \ref{thm:Hs_equality_open}.
\end{proof}

We also note the following result, which follows from \rf{isdual}. 
\begin{lem}
Let $\Omega\subset \R^n$ be non-empty and open and let $s\in\R$. Then
 \begin{align*}
\label{}
\tilde H^s(\Omega) = H^s_{\overline{\Omega}}   
\; \mbox{ if and only if }\; 
\tilde H^{-s}(\R^n\setminus \overline{\Omega}) = H^{-s}_{\R^n \setminus \Omega} .
\end{align*}
\end{lem}

\begin{rem}
\label{rem:ss_equality}
While $\tilde H^s(\Omega)$ and $H^s_{\overline{\Omega}}$ may not coincide in general, we remark that for any non-empty open subset $\Omega\subset\R^n$ and any $s\in\R$ it holds (cf.\ \cite[Lemma 3.24]{McLean}) that, for arbitrary $\eps>0$, $H^s_{\overline\Omega}\subset \tilde{H}^s(\Omega_\eps)$, where $\Omega_\eps:=\{\bx\in\R^n:\dist(\bx,\Omega)<\eps \}$. In other words, any element of $H^s_{\overline\Omega}$ can be approximated to arbitrary precision by a smooth function whose support lies within an arbitrarily small neighbourhood of $\Omega$. This fact will underpin the rigorous definition of the layer potentials discussed later in the paper.
\end{rem}

\section{Screen scattering problems}
\label{ScatProb}
Having established our Sobolev space notation, we now turn to problem (i) of \S\ref{sec:Intro}, namely that of correctly formulating problems of acoustic scattering by arbitrary planar screens. 
In what follows let $\Gamma$ be a bounded and relatively open non-empty subset of $\Gamma_\infty:=\{ \bx\in\mathbb{R}^n:x_n=0\}$, %
and  %
let $D:=\R^n\setminus \overline\Gamma$. 
 We consider in this section problems of acoustic scattering in which a wave is incident from the unbounded domain $D$ onto the planar screen $\Gamma$.
We will seek solutions to our scattering problems in the natural energy space $W^1_{\mathrm{loc}}(D)$, imposing the boundary conditions through trace operators from this space to Sobolev spaces defined on $\Gamma$.

\subsection{Function spaces and trace operators}
\label{subsec:FuncSpacesTrace}

To define Sobolev spaces on $\Gamma_\infty$ and on the screen $\Gamma$, we make the natural associations of $\Gamma_\infty$ with $\R^{n-1}$ and of $\Gamma$ with $\tGamma:=\{ \tbx\in \R^{n-1}:(\tbx,0)\in\Gamma\}\subset \R^{n-1}$ and set $H^s(\Gamma_\infty) := H^s(\R^{n-1})$, $H^{s}(\Gamma):=H^{s}(\tGamma)$, $\tilde{H}^{s}(\Gamma):=\tilde{H}^{s}(\tGamma)$, and $H^{s}_{\overline{\Gamma}}:=H^{s}_{\overline{\tGamma}}$ (with $C^\infty(\Gamma_\infty)$, $\scrD(\Gamma_\infty)$, $\scrD(\overline\Gamma)$ and $\scrD(\Gamma)$ defined analogously). Let $U^+$ and $U^-$ denote the upper and lower half-spaces, respectively, i.e., $U^+:= \{\bx=(\tbx,x_n):x_n>0\}$ and $U^- := \R^n\setminus \overline{U^+}$. We define (Dirichlet) trace operators $\gamma^\pm:\scrD(\overline U^\pm)\to \scrD(\Gamma_\infty)$ by $\gamma^\pm u := u|_{\Gamma_\infty}$. It is well know that these trace operators extend to bounded linear operators $\gamma^\pm:W^1(U^\pm)\to H^{1/2}(\Gamma_\infty)$. Of note is the fact that%
\begin{align}
\label{W1DCharac}
W^1(D) = \{u\in L^2(D): u|_{U^\pm}\in W^1(U^\pm) \textrm{ and } \gamma^+u=\gamma^-u \textrm{ on } \Gamma_\infty\setminus \overline{\Gamma}\}.
\end{align}
Similarly, we define normal derivative operators $\partial_{\bn}^\pm:\scrD(\overline U^\pm)\to \scrD(\Gamma_\infty)$ by
$$
\partial_{\bn}^\pm u(\bx) = \frac{\partial u}{\partial x_n} (\bx), \quad \bx \in \Gamma_\infty.
$$
It is well known that the domain of these operators can be extended to $W^1(U^\pm;\Delta):= \{u\in H^1(U^\pm):\Delta u\in L^2(U^\pm)\}$, where $\Delta u$ is the (weak) Laplacian of $u$, as mappings $\partial_\bn^\pm:W^1(U^\pm;\Delta)\to H^{-1/2}(\Gamma_\infty)=(H^{1/2}(\Gamma_\infty))^*$, defined by
$$
\langle \partial_\bn^\pm u, \phi \rangle_{H^{-1/2}(\Gamma_\infty)\times H^{1/2}(\Gamma_\infty)} = \mp \int_{U^\pm} \left( \nabla u \cdot \nabla \bar v + \bar v \Delta u \right) \rd x, \quad u\in W^1(U^\pm;\Delta), \; \phi\in H^{1/2}(\Gamma_\infty),
$$
with $v\in W^1(U^\pm)$ with $\gamma^\pm v = \phi$, and that Green's first identity holds, that
$$
\langle \partial_\bn^\pm u, \gamma^\pm v\rangle_{H^{-1/2}(\Gamma_\infty)\times H^{1/2}(\Gamma_\infty)} = \mp \int_{U^\pm} \left( \nabla u \cdot \nabla \bar v + \bar v \Delta u \right) \rd x, \quad u\in W^1(U^\pm;\Delta), \; v\in W^{1}(U^\pm).
$$

Let $
W:= \{u\in L^2_{\rm loc}(\R^n): u|_{U^\pm}\in W_{\rm loc}^1(U^\pm;\Delta)\}
$.
Then for $u\in W$ satisfying
\begin{align*}
\label{}
\gamma^+(\chi u)|_{\Gamma_\infty\setminus\overline\Gamma} - \gamma^+(\chi u)|_{\Gamma_\infty\setminus\overline\Gamma} = 0, \quad \mbox{for all } \chi\in\scrD(\R^n),
\end{align*}
we define
\begin{align*}
\label{}
[u]&:=\gamma^+(\chi u)-\gamma^-(\chi u)\in H^{1/2}_{\overline{\Gamma}},
\end{align*}
where $\chi$ is any element of $\scrD_{1,\Gamma}(\R^n):=\{\phi \in \scrD(\R^n)$: $\phi=1$ in some neighbourhood of $\Gamma\}$.
Similarly, for $u\in W$ satisfying
\begin{align*}
\label{}
\partial_\bn^+(\chi u)|_{\Gamma_\infty\setminus\overline\Gamma} - \partial_\bn^+(\chi u)|_{\Gamma_\infty\setminus\overline\Gamma} = 0, \quad \mbox{for all } \chi\in\scrD(\R^n),
\end{align*}
we define
\begin{align*}
\label{}
[\pdonetext{u}{\bn}]&:=\partial^+_\bn(\chi u)-\partial^-_\bn(\chi u)\in H^{-1/2}_{\overline{\Gamma}},
\end{align*}
where $\chi$ is any element of $\scrD_{1,\Gamma}(\R^n)$. 

\subsection{Layer potentials and boundary integral operators}

We define the single and double layer potentials
\begin{align*}
\label{}
\cS_k:H^{-1/2}_{\overline{\Gamma}}\to C^2(D)\cap W^1_{\rm loc}(D),\qquad \cD_k:H^{1/2}_{\overline{\Gamma}}\to C^2(D)\cap W^1_{\rm loc}(D),
\end{align*}
by
\begin{align*}
\label{}
\cS_k\phi (\bx)&:=\left\langle \gamma^\pm(\rho\Phi(\bx,\cdot)),\overline{\phi}\right\rangle_{H^{1/2}(\Gamma_\infty)\times H^{-1/2}(\Gamma_\infty)}, \qquad \bx\in D,\, \phi\in H^{-1/2}_{\overline{\Gamma}},\\
\cD_k\psi (\bx)&:=\left\langle \psi, \overline{\partial^\pm_\bn(\rho\Phi(\bx,\cdot))}\right\rangle_{H^{1/2}(\Gamma_\infty)\times H^{-1/2}(\Gamma_\infty)}, \qquad \bx\in D,\, \psi\in H^{1/2}_{\overline{\Gamma}},
\end{align*}
$\rho$ is any element of $\scrD_{1,\Gamma}(\R^n)$ with $\bx\not\in \supp{\rho}$, and
\begin{align}
\label{FundSoln}
\Phi(\bx,\by) :=
\begin{cases}
\dfrac{\re^{\ri k|\bx-\by|}}{4\pi |\bx-\by|}, & n=3,\\[3mm]
\dfrac{\ri }{4}H_0^{(1)}(k|\bx-\by|), & n=2,
\end{cases}
\qquad \bx,\by \in \R^n,
\end{align}
is the fundamental solution of the Helmholtz equation.

The following properties of $\cS_k$ and $\cD_k$ are well-known when the densities lie in $\tilde H^{-1/2}(\Gamma)$ and $\tilde H^{-1/2}(\Gamma)$ respectively. The extension to $H^{\pm 1/2}_{\overline{\Gamma}}$ can be carried out with the help of Remark \ref{rem:ss_equality}.
\begin{thm}
\label{LayerPotRegThm}
(i) For any $\phi\in H^{-1/2}_{\overline{\Gamma}}$ and $\psi\in H^{1/2}_{\overline{\Gamma}}$ the potentials $\cS_k\phi $ and $\cD_k\psi $ are twice-continuously differentiable in $D$, satisfy the Helmholtz equation in $D$, and satisfy the Sommerfeld radiation condition at infinity;

(ii) for any $\chi\in \scrD(\R^n)$ the following mappings are bounded:
\begin{align*}
\label{}
\chi\cS_k:H^{-1/2}_{\overline{\Gamma}}\to W^1(D),\qquad
\chi \cD_k:H^{1/2}_{\overline{\Gamma}}\to  W^1(D);
\end{align*}
(iii) the following jump relations hold for all $\phi\in H^{-1/2}_{\overline{\Gamma}}$,  $\psi\in H^{1/2}_{\overline{\Gamma}}$ and $\chi\in \scrD_{1,\Gamma}(\R^n)$:
\begin{align}
\label{JumpRelns1}
[\cS_k\phi]&=0,\\
\label{JumpRelns2}
\partial^\pm_\bn(\chi \cS_k\phi) &= \mp\phi/2, \quad \mbox{so that } [\pdonetext{(\cS_k\phi)}{\bn}]=-\phi, \\
\label{JumpRelns3}
\gamma^\pm(\chi \cD_k\psi) &=\pm \psi/2, \quad \mbox{so that }[\cD_k\psi]=\psi,\\ %
\label{JumpRelns4}
[\pdonetext{(\cD_k\psi)}{\bn}]&=0;
\end{align}
(iv) for $\phi\in \scrD(\Gamma)$ the following integral representations are valid:
\begin{align}
\label{SkPotIntRep}
\cS_k\phi (\bx) &= \int_\Gamma \Phi(\bx,\by) \phi(\by)\, \rd s(\by), \qquad \bx\in D,\\
\label{DkPotIntRep}
\cD_k\phi (\bx) &= \int_\Gamma \pdone{\Phi(\bx,\by)}{\bn(\by)} \phi(\by)\, \rd s(\by), \qquad \bx\in D.
\end{align}
\end{thm}

\begin{proof}(Sketch) The proof mostly follows standard arguments, but the extension to $H^{\pm 1/2}_{\overline{\Gamma}}$ requires us to use the fact (cf.\ Remark \ref{rem:ss_equality}) that $H^{\pm  1/2}_{\overline{\Gamma}}\subset \tilde{H}^{ \pm 1/2}(\Gamma_\eps)$ for any neighbourhood $\Gamma_\eps$ of $\Gamma$, so that elements of $H^{\pm 1/2}_{\overline{\Gamma}}$ can be approximated arbitrarily well by smooth functions whose support is arbitrarily close to $\Gamma$.  

To show the claimed regularity, one uses
continuity of the trace operators to show that the potentials converge uniformly on compact subsets of $D$, and that they are bounded in terms of the norm of their argument. The potentials are clearly infinitely differentiable in $D$ and satisfy the Helmholtz equation and the SRC for densities $\phi\in \scrD(\Gamma_\eps)$, then one can apply standard elliptic regularity results (e.g.\ \cite[Lemma 3.9]{GilbargTrudinger}), along with the uniform boundedness, to deduce that the same is also true for general densities. To show that the potentials are in $W^1_{\rm loc}(D)$, one first shows that they are in $W^1_{\rm loc}(U^\pm)$. To do this, one shows that the potentials, after multiplication by a cut-off function, live in $H^1(U^\pm)=\tilde{H}^{-1}(U^\pm)^*$. To show that they define elements of $\tilde{H}^{-1}(U^\pm)^*$, first consider a smooth argument, then use Fubini's Theorem to rewrite the duality pairing in terms of the trace operator and the Newton potential. Finally, use the fact that the Newton potential maps $H^{-1}(\R^n)$ to $H^1(\R^n)$. For the double layer potential one has to use \cite[Lemma 4.3]{McLean} to get a bound on the Neumann trace in a different norm to the one on $W^1(U^\pm;\Delta)$. Once boundedness has been shown, one can then extend to non-smooth densities. Since we first work in $U^\pm$ rather than the whole of $D$, the modification of the proof from $\tilde{H}^{\pm 1/2}(\Gamma_\eps)$ to $H^{\pm 1/2}_{\overline{\Gamma}}$ is straightforward. One then appeals to smoothness across $\Gamma_\infty\setminus \overline{\Gamma}$ to conclude that the potentials map into $W^1_{\rm loc}(D)$ using \rf{W1DCharac}. 

The jump relations and integral representations are standard and can be derived from the corresponding results for Lipschitz domains (see, e.g., \cite{ChGrLaSp:11}). We also make use of the fact that the double layer and adjoint double layer operators vanish on a flat screen. %
\end{proof}

We then have the following version of Green's representation theorem for the screen:
\begin{thm}%
\label{GreenRepThm}
Let $u\in C^2(D)\cap W^1_{\rm loc}(D)$ with $(\Delta + k^2)u=0$ in $D$ and suppose that $u$ satisfies the Sommerfeld radiation condition at infinity. Then
\begin{align}
\label{RepThmGeneral}
u(\bx) = -\cS_k\left[\pdonetext{u}{\bn}\right](\bx) + \cD_k[u](\bx), \qquad\bx\in D.
\end{align}
\end{thm}
\begin{proof}
First let $\bx\in U^+\cup U^-$. Apply the standard Green's representation theorem \cite[Theorem 2.20]{ChGrLaSp:11} in the Lipschitz domains $U^\pm_R:=U^\pm \cap B_R(0)$, where $R>0$ is large enough such that $\overline{\Gamma}\subset B_R(0)$. Summing the resulting equations gives $u(\bx)$ as a sum of layer potentials defined over the two hemispherical boundaries $\partial U^\pm_R$. The contributions from $\Gamma_\infty\cap B_R(0)\setminus \overline{\Gamma}$ cancel because $u\in C^2(D)$, so that $u(\bx)$ is a sum of layer potentials on $\partial B_R(0)$ and on $\overline{\Gamma}$. The formula \rf{RepThmGeneral} then follows from letting $R\to \infty$, with the contribution from $\partial B_R(0)$ tending to zero because $u$ satisfies the radiation condition (cf.~\cite{CoKr:83}). The extension to $\bx\in \Gamma_\infty\setminus \overline{\Gamma}$ follows by continuity.
\end{proof}

We also define the single-layer and hypersingular boundary integral operators
\begin{align*}
\label{}
S_k:H^{-1/2}_{\overline{\Gamma}}\to H^{1/2}(\Gamma),\qquad T_k:H^{1/2}_{\overline{\Gamma}}\to H^{-1/2}(\Gamma),
\end{align*}
by
\begin{align*}
S_k\phi &:=\gamma^\pm(\chi\cS_k\phi)|_\Gamma, \qquad \phi\in H^{-1/2}_{\overline{\Gamma}},\\
T_k\phi &:=\partial^\pm_\bn(\chi\cD_k\psi)|_\Gamma, \qquad \psi\in H^{1/2}_{\overline{\Gamma}},
\end{align*}
where $\chi$ is any element of $\scrD_{1,\Gamma}(\R^n)$, and either of the $\pm$ traces may be taken (cf.\ \rf{JumpRelns1} and \rf{JumpRelns4}).

For $\phi\in \scrD(\Gamma)$ the following integral representations are valid:
\begin{align}
\label{SDef}
S_k\phi (\bx) &= \int_\Gamma \Phi(\bx,\by) \phi(\by)\, \rd s(\by), \qquad \bx\in \Gamma,\\
\label{TDef}
T_k\phi (\bx) &= \pdone{}{\bn(\bx)}\int_\Gamma \pdone{\Phi(\bx,\by)}{\bn(\by)} \phi(\by)\, \rd s(\by), \qquad \bx\in \Gamma.
\end{align}

\subsection{Boundary value problems}

We now recall the boundary value problems $\sD$ and $\sN$ introduced in Definitions \ref{def:SPD} and \ref{def:SPN}; for ease of reference we restate them here, with the boundary conditions \rf{eqn:bc1} and \rf{eqn:bc2} stated more precisely in terms of traces: 
\begin{defn}[Problem $\sD$]
Given $g_{\sD}\in H^{1/2}(\Gamma)$, find $u\in C^2\left(D\right)\cap  W^1_{\mathrm{loc}}(D)$ such that
\begin{align}
  \Delta u+k^2u  &=  0, \quad \;\; \mbox{in }D, \label{eqn:HE1restate} \\
 \gamma^\pm (\chi u)\vert_\Gamma &=g_{\sD}, \quad \mbox{for any }\chi\in\scrD_{1,\Gamma}(\R^n),\label{eqn:bc1restate}
\end{align}
and $u$ satisfies the Sommerfeld radiation condition. 
\end{defn}
\begin{defn}[Problem $\sN$]
Given $g_{\sN}\in H^{-1/2}(\Gamma)$, find $u\in C^2\left(D\right)\cap W^1_{\mathrm{loc}}(D)$ such that
\begin{align}
  \Delta u+k^2u  &=  0, \quad\;\; \mbox{in }D, & \label{eqn:HE2restate} \\
\partial_\bn^\pm (\chi u)\vert_\Gamma &=g_{\sN}, \quad \mbox{for any }\chi\in\scrD_{1,\Gamma}(\R^n),& \label{eqn:bc2restate}%
\end{align}
and $u$ satisfies the Sommerfeld radiation condition. 
\end{defn}

Following the standard direct boundary integral equation approach, we would like to use Theorem \ref{GreenRepThm} to represent the solution of problem $\sD$ (assuming it exists) in the form
\begin{align}
\label{eqn:SLPRep}
u(\bx )= -\cS_k\left[\pdonetext{u}{\bn}\right](\bx), \qquad\bx\in D.
\end{align}
Of course, deriving \rf{eqn:SLPRep} from the general representation formula \rf{RepThmGeneral} requires us to show that 
\begin{align}
\label{eqn:ujumpvanishes}
[u]=0. 
\end{align}
Now, by definition we have that $[u]\in H^{1/2}_{\overline\Gamma}$, and the boundary condition \rf{eqn:bc1restate} implies that $[u]|_\Gamma=0$. So it must hold that $[u]\in H^{1/2}_{\partial\Gamma}$. If $\partial\Gamma$ is $1/2$-null, i.e.\ $H^{1/2}_{\partial\Gamma}=\{0\}$, then \rf{eqn:ujumpvanishes} immediately follows. This holds, for example, if $\Gamma$ is $C^0$ (in particular if $\Gamma$ is Lipschitz), by Lemma \ref{lem:polarity}\rf{kk0}. But if $\partial\Gamma$ is not $1/2$-null then any non-zero $\psi\in H^{1/2}_{\partial\Gamma}\subset H^{1/2}_{\overline\Gamma}$ provides, by Theorem \ref{LayerPotRegThm}, a non-trivial solution (namely $\cD_k\psi $) of the homogeneous Dirichlet problem (i.e.\ problem $\sD$ with $g_{\sD}=0$). So the solution to problem $\sD$ for general $\Gamma$ and general $g_{\sD}\in H^{1/2}(\Gamma)$ is not unique. 

Similarly, we would like to represent the solution of problem $\sN$ (assuming it exists) in the form
\begin{align}
\label{eqn:DLPRep}
 u(\bx ) =  \cD_k[u](\bx), \qquad\bx\in D,
\end{align}
and deriving \rf{eqn:DLPRep} from \rf{RepThmGeneral} requires us to show that 
\begin{align}
\label{eqn:dudnjumpvanishes}
[\pdonetext{u}{\bn}]=0. 
\end{align}
By definition we have that $[\pdonetext{u}{\bn}]\in H^{-1/2}_{\overline\Gamma}$, and the boundary condition \rf{eqn:bc2restate} gives $[\pdonetext{u}{\bn}]|_\Gamma=0$. So it must hold that $[\pdonetext{u}{\bn}]\in H^{-1/2}_{\partial\Gamma}$. If $\partial\Gamma$ is $(-1/2)$-null, i.e.\ $H^{-1/2}_{\partial\Gamma}=\{0\}$, then \rf{eqn:dudnjumpvanishes} immediately follows. This holds, for example, if $\Gamma$ is Lipschitz, by Lemma \ref{lem:polarity}\rf{kk}. But if $\partial\Gamma$ is not $(-1/2)$-null then any non-zero $\phi\in H^{-1/2}_{\partial\Gamma}\subset H^{-1/2}_{\overline\Gamma}$ provides, by Theorem \ref{LayerPotRegThm}, a non-trivial solution (namely $\cS_k\phi$) of the homogeneous Neumann problem (i.e.\ problem $\sN$ with $g_{\sN}=0$). So the solution to problem $\sN$ for general $\Gamma$ and general $g_{\sN}\in H^{-1/2}(\Gamma)$ is not unique. 

To deal with this possible nonuniqueness we modify the BVPs $\sD$ and $\sN$ by requiring that their solutions satisfy the jump conditions \rf{eqn:ujumpvanishes} and \rf{eqn:dudnjumpvanishes}, respectively. 

\begin{defn}
Let $\sD'$ and $\sN'$ denote problems $\sD$ and $\sN$, supplemented respectively with the additional constraints \rf{eqn:ujumpvanishes} and \rf{eqn:dudnjumpvanishes}.
\end{defn}

It then follows straight from Theorems \ref{LayerPotRegThm} and \ref{GreenRepThm} that the modified BVPs $\sD'$ and $\sN'$ are equivalent to the usual boundary integral equations involving the operators $S_k$ and $T_k$, respectively.
\begin{defn}[Problem $\sS$]
Given $g_{\sD}\in H^{1/2}(\Gamma)$, find $\phi\in H^{-1/2}_{\overline{\Gamma}}$ such that
\begin{align}
\label{BIE_sl}
-S_k\phi= g_{\sD}.
\end{align}
\end{defn}
\begin{defn}[Problem $\sT$]
Given $g_{\sN}\in H^{-1/2}(\Gamma)$, find $\psi \in H^{1/2}_{\overline{\Gamma}}$ such that
\begin{align}
\label{BIE_hyp}
T_k\psi = g_{\sN}.
\end{align}
\end{defn}

\begin{thm}
\label{DirEquivThm}
Suppose that $u$ is a solution of problem $\sD'$. Then the representation formula 
\rf{eqn:SLPRep}
holds, and $[\pdonetext{u}{\bn}]\in H^{-1/2}_{\overline{\Gamma}}$ satisfies problem $\sS$. Conversely, suppose that $\phi\in H^{-1/2}_{\overline{\Gamma}}$ satisfies problem $\sS$. Then $u:=-\cS_k\phi$ satisfies problem $\sD'$, and $[\pdonetext{u}{\bn}]=\phi$.
\end{thm}

\begin{thm}
\label{NeuEquivThm}
Suppose that $u$ is a solution of problem $\sN'$. Then the representation formula 
\rf{eqn:DLPRep}
holds, and $[u]\in H^{1/2}_{\overline{\Gamma}}$ satisfies problem $\sT$. Conversely, suppose that $\psi\in H^{1/2}_{\overline{\Gamma}}$ satisfies problem $\sT$. Then $u:=\cD_k\psi$ satisfies problem $\sN'$, and $[u]=\psi$.
\end{thm}

The question of the unique solvability (or otherwise) of problems $\sS$ and $\sT$, and hence (by Theorems \ref{DirEquivThm} and \ref{NeuEquivThm}) of problems $\sD'$ and $\sN'$, is answered by the following two theorems, which follow from Theorems \ref{ThmSNormSmooth}, \ref{ThmCoerc}, \ref{ThmTNormSmooth} and \ref{ThmTCoercive} (where the dependence of the continuity and coercivity constants on both $k$ and $\Gamma$ is stated explicitly, the continuity results being shown to hold on Sobolev spaces of arbitrary real index $s$). 
We emphasize that these results all hold with $\Gamma$ an arbitrary non-empty open subset of $\Gamma_\infty$. 

\begin{thm}
\label{ThmSk}
For every $k>0$ the single-layer operator $S_k:H^{-1/2}_{\overline\Gamma} \to H^{1/2}(\Gamma)$ is continuous, and is coercive as an operator $S_k:\tilde{H}^{-1/2}(\Gamma) \to H^{1/2}(\Gamma) \cong (\tilde{H}^{-1/2}(\Gamma))^*$.
\end{thm}

\begin{thm}
\label{ThmTk}
For every $k>0$ the hypersingular operator $T_k:H^{1/2}_{\overline\Gamma} \to H^{-1/2}(\Gamma)$ is continuous, and is coercive as an operator $T_k:\tilde{H}^{1/2}(\Gamma) \to H^{-1/2}(\Gamma) \cong (\tilde{H}^{1/2}(\Gamma))^*$.
\end{thm}

Thus, by the Lax-Milgram Lemma, the operator $S_k:\tilde{H}^{-1/2}(\Gamma) \to H^{1/2}(\Gamma)$ is invertible. Hence if $\tilde{H}^{-1/2}(\Gamma) = H^{-1/2}_{\overline\Gamma}$, then problem $\sS$ (and hence also problem $\sD'$) is uniquely solvable. This holds, for example, if $\Gamma$ is $C^0$, by Lemma \ref{lem:sob_equiv}. But if $ \tilde{H}^{-1/2}(\Gamma)\subsetneqq H^{-1/2}_{\overline\Gamma}$ then $S_k:H^{-1/2}_{\overline\Gamma} \to H^{1/2}(\Gamma)$ is surjective but not injective, i.e., a solution to problem $\sS$ (and hence also problem $\sD'$) exists, but this solution is not unique.

Similarly, the operator $T_k:\tilde{H}^{1/2}(\Gamma) \to H^{-1/2}(\Gamma)$ is invertible. Hence if $\tilde{H}^{1/2}(\Gamma) = H^{1/2}_{\overline\Gamma}$, then problem $\sT$ (and hence also problem $\sN'$) is uniquely solvable. This holds, for example, if $\Gamma$ is $C^0$, by Lemma \ref{lem:sob_equiv}. But if $ \tilde{H}^{1/2}(\Gamma)\subsetneqq H^{1/2}_{\overline\Gamma}$ then $T_k:H^{1/2}_{\overline\Gamma} \to H^{-1/2}(\Gamma)$ is surjective but not injective, i.e., a solution to problem $\sT$ (and hence also problem $\sN'$) exists, but this solution is not unique.

In order to guarantee unique solvability for arbitrary non-empty open $\Gamma$ we must modify the BVPs $\sD'$ and $\sN'$ further, to require that the solutions of the integral equations lie in the spaces $\tilde{H}^{-1/2}(\Gamma)$ and $\tilde{H}^{1/2}(\Gamma)$ respectively.

\begin{defn}[Problem $\sD''$]
Given $g_{\sD}\in H^{1/2}(\Gamma)$, find $u\in C^2\left(D\right)\cap  W^1_{\mathrm{loc}}(D)$ such that
\begin{align}
  \Delta u+k^2u  &=  0, \quad \;\; \mbox{in }D, \label{eqn:HE1restate2} \\
 \gamma^\pm (\chi u)\vert_\Gamma &=g_{\sD}, \quad \mbox{for any }\chi\in\scrD_{1,\Gamma}(\R^n),\label{eqn:bc1restate2} \\
 [u] & = 0, \\
 [\pdonetext{u}{\bn}] &\in \tilde H^{-1/2}(\Gamma), 
\end{align}
and $u$ satisfies the Sommerfeld radiation condition. 
\end{defn}
\begin{defn}[Problem $\sN''$]
Given $g_{\sN}\in H^{-1/2}(\Gamma)$, find $u\in C^2\left(D\right)\cap W^1_{\mathrm{loc}}(D)$ such that
\begin{align}
  \Delta u+k^2u  &=  0, \quad\;\; \mbox{in }D, & \label{eqn:HE2restate2} \\
\partial_\bn^\pm (\chi u)\vert_\Gamma &=g_{\sN}, \quad \mbox{for any }\chi\in\scrD_{1,\Gamma}(\R^n),& \label{eqn:bc2restate2} \\
 [\pdonetext{u}{\bn}] & = 0, \\
 [u] &\in \tilde H^{1/2}(\Gamma),
\end{align}
and $u$ satisfies the Sommerfeld radiation condition. 
\end{defn}

We then have the following well-posedness results, which hold for an arbitrary non-empty open $\Gamma$.

\begin{thm}
\label{DirExUn}
For any $g_{\sD}\in H^{1/2}(\Gamma)$ problem $\sD''$ has a unique solution given by formula \rf{eqn:SLPRep}, where $[\pdonetext{u}{\bn}]$ is the unique solution in $\tilde H^{-1/2}(\Gamma)$ of equation \rf{BIE_sl}.
\end{thm}
\begin{thm}
\label{NeuExUn}
For any $g_{\sN}\in H^{-1/2}(\Gamma)$ problem $\sN''$ has a unique solution given by formula \rf{eqn:DLPRep}, where $[u]$ is the unique solution in $\tilde H^{1/2}(\Gamma)$ of equation \rf{BIE_hyp}.
\end{thm}

We remark that if $\tilde{H}^{-1/2}(\Gamma) = H^{-1/2}_{\overline\Gamma}$ (e.g.\ if $\Gamma$ is $C^0$) then problems $\sD''$ and $\sD'$ are equivalent; if $\partial\Gamma$ is $1/2$-null (e.g.\ if $\Gamma$ is $C^0$) then problems $\sD'$ and $\sD$ are equivalent. So, in particular, if $\Gamma$ is $C^0$ all three problems are equivalent, and the original BVP $\sD$ is uniquely solvable.

Similarly, if $\tilde{H}^{1/2}(\Gamma) = H^{1/2}_{\overline\Gamma}$ (e.g.\ if $\Gamma$ is $C^0$) then problems $\sN''$ and $\sN'$ are equivalent; if $\partial\Gamma$ is $(-1/2)$-null (e.g.\ if $\Gamma$ is Lipschitz) then problems $\sN'$ and $\sN$ are equivalent. So, in particular, if $\Gamma$ is Lipschitz all three problems are equivalent, and the original BVP $\sN$ is uniquely solvable.

\section{Fourier representations for layer potentials and BIOs}
\label{sec:FourierRep}
Our approach to proving the continuity and coercivity results in Theorems \ref{ThmSk} and \ref{ThmTk}, and furthermore in determining the dependence of the associated continuity and coercivity constants on $k$ and $\Gamma$ (which we do in \S\ref{sec:SkAnalysis} and \S\ref{sec:TkAnalysis}), is to make use of the fact that, because the screen is flat, the single and double layer potentials can be expressed in terms of Fourier transforms (at least for sufficiently smooth densities). As a result, the single-layer and hypersingular BIOs can be thought of as pseudodifferential operators, as the following theorem shows. 
We note that the parts of Theorem \ref{ThmFourierRep} relating to $T_k$ were stated and proved (via a slightly different method to that presented here) for the case $n=3$ and $\partial\Gamma$ smooth in \cite[Theorems 1 and 2]{Ha-Du:92}. %
\begin{thm}
\label{ThmFourierRep}
Let $\phi\in \scrD (\Gamma)$. Then
\begin{align}
\label{SPotFourierRep}
\cS_k\phi(\bx)&=\frac{\ri}{2(2\pi)^{(n-1)/2}}\int_{\R^{n-1}}\frac{\re^{\ri (\bxi\cdot \tbx+|x_n| Z(\bxi))}}{Z(\bxi)}\widehat{\varphi}(\bxi)\,\rd \bxi,  & &\bx=(\tbx,x_n)\in D,\\
\label{DPotFourierRep}
\cD_k\phi(\bx)&=\frac{\sign{x_n}}{2(2\pi)^{(n-1)/2}}\int_{\R^{n-1}}\re^{\ri (\bxi\cdot \tbx+|x_n| Z(\bxi))}\widehat{\varphi}(\bxi)\,\rd \bxi, & & \bx=(\tbx,x_n)\in D,
\end{align}
where $\hat{}$ represents the Fourier transform with respect to $\tbx\in \R^{n-1}$ and
\begin{align}
\label{ZDef}
Z(\bxi):=
\begin{cases}
\sqrt{k^2-|\bxi|^2}, & |\bxi|\leq k\\
\ri\sqrt{|\bxi|^2-k^2}, & |\bxi|> k,
\end{cases}
\qquad \bxi\in \R^{n-1}.
\end{align}
The operators $S_k,T_k:\scrD (\Gamma) \to \scrD (\overline\Gamma)$ satisfy $S_k\phi = (S_k^\infty \phi)|_\Gamma$ and $T_k\phi = (T_k^\infty \phi)|_\Gamma$, where $S_k^\infty, T_k^\infty:\scrD(\R^{n-1})\to C^\infty(\R^{n-1})$ are the pseudodifferential operators defined for $\varphi\in \scrD(\R^{n-1})$ by
\begin{align}
\label{SFourierRep}
S_k^\infty\varphi(\tbx)&=\frac{\ri}{2(2\pi)^{(n-1)/2}}\int_{\R^{n-1}}\frac{\re^{\ri \bxi\cdot \tbx}}{Z(\bxi)}\widehat{\varphi}(\bxi)\,\rd \bxi, & & \tbx\in\R^{n-1},\\
\label{TFourierRep}
T_k^\infty\varphi(\tbx)&=\frac{\ri}{2(2\pi)^{(n-1)/2}}\int_{\R^{n-1}}Z(\bxi)\re^{\ri \bxi\cdot \tbx}\widehat{\varphi}(\bxi)\,\rd \bxi, & & \tbx\in\R^{n-1}.%
\end{align}
Furthermore, for $\phi, \psi\in \scrD(\Gamma)$ we have that
\begin{align}
\label{SIPFourierRep}
(S_k\phi,\psi)_{L^2(\Gamma)} &= \frac{\ri}{2 }\int_{\R^{n-1}} \frac{1}{Z(\bxi)} \widehat{\phi}(\bxi) \overline{\widehat{\psi}(\bxi)} \,\rd \bxi,\\
\label{TIPFourierRep}
(T_k\phi,\psi)_{L^2(\Gamma)} &= \frac{\ri}{2 }\int_{\R^{n-1}} Z(\bxi) \widehat{\phi}(\bxi) \overline{\widehat{\psi}(\bxi)} \,\rd \bxi.
\end{align}
\end{thm}

\begin{proof}
From the integral representation \rf{SkPotIntRep}, we see that $\cS_k\phi$ can be expressed as
\begin{align}
\label{}
\cS_k\phi(\bx) = (\Phi_c(\cdot,x_n)\ast \phi)(\tbx),
\end{align}
where $\ast$ indicates a convolution over $\R^{n-1}$ (with $x_n$ treated as a parameter) and
\begin{align*}
\label{}
\Phi_c(\tbx,x_n) :=\Phi((\tbx ,x_n),\bs{0}) =
\begin{cases}
\dfrac{\re^{\ri k\sqrt{r^2+x_n^2}}}{4\pi \sqrt{r^2+x_n^2}}, & n=3,\\[3mm]
\dfrac{\ri }{4}H_0^{(1)}(k\sqrt{r^2+x_n^2}), & n=2,
\end{cases}
\qquad  r=|\tbx |,\, \tbx \in\R^{n-1}.
\end{align*}
Hence the Fourier transform (with respect to $\tbx\in\R^{n-1}$) of $\cS_k \phi$ is given by the product
\begin{align*}
\widehat{\cS_k \phi}(\bxi,x_n) = (2\pi)^{(n-1)/2}\,\widehat{\Phi_c}(\bxi,x_n)\hat{\varphi}(\bxi).
\end{align*}
To evaluate $\widehat{\Phi_c}$ we note that for a function $f(\bx )=F(r)$, where $r=|\bx |$ for $\bx \in\R^{d}$, $d=1,2$, the Fourier transform of $f$ is given by (cf.\ \cite[\S{B.5}]{Grafakos})
\footnote{Strictly speaking, \cite[\S{B.5}]{Grafakos} only provides \rf{FTRadial} for the case where $f\in L^1(\R^d)$, while we are in the case where $f\in L^1_{\rm loc}(\R^d)$ with some decay at infinity, but not enough to be in $L^1(\R^d)$. 
But for the functions $f=\Phi_c(\cdot,x_n)$ one can check using the Dominated Convergence Theorem for Lebesgue integrals that the formulas \rf{FTRadial} do indeed coincide with the distributional Fourier transforms, but we do not supply further details here. 
For the case $d=2$ one should also note the related integral representation for the Bessel function in \cite[(10.9.2)]{DLMF}.}
\begin{align}
\label{FTRadial}
\hat{f}(\bxi)=
\begin{cases}
\displaystyle{\int_0^\infty F(r)J_0(|\bxi|r)r\,\rd r}, & d=2,\\[3mm]
\displaystyle{
\sqrt{\frac{2}{\pi}}\int_0^\infty F(r)\cos(\bxi r)\,\rd r}, & d=1.
\end{cases}
\end{align}
This result, combined with the identies \cite[(6.677), (6.737)]{Ryzhik}
and \cite[(10.16.1), (10.39.2)]{DLMF},
gives (see also~\cite[eqn (4.17)]{ChHePo06} for the case $n=3$)
\begin{align*}
\label{}
\widehat{\Phi_c}(\bxi,x_n)= \dfrac{\ri \,\re^{\ri |x_n| Z(\bxi)}}{2(2\pi)^{(n-1)/2}Z(\bxi)},
\end{align*}
where $Z(\bxi)$ is defined as in \rf{ZDef}. The representation \rf{SFourierRep} is then obtained by Fourier inversion.

The representation \rf{DPotFourierRep} for $\cD_k\phi$ can be then obtained from \rf{SPotFourierRep} by noting that
\begin{align*}
\label{}
\pdone{\Phi(\bx,\by)}{\bn(y)}=\pdone{\Phi(\bx,\by)}{y_n} = -\pdone{\Phi(\bx,\by)}{x_n}, \qquad \bx\in D, \,\by\in \Gamma,
\end{align*}
and the representations for $S_k$ and $T_k$ follow from taking the appropriate traces of \rf{SPotFourierRep} and \rf{DPotFourierRep}.%

Finally, \rf{SIPFourierRep} and \rf{TIPFourierRep} follow from viewing $S^\infty_k\phi$ and $T^\infty_k\phi$ as elements of $C^\infty(\R^{n-1})\cap \mathscr{S}^*(\R^{n-1})$ and recalling the definition \rf{FTDistDef} of the Fourier transform of a distribution, e.g.%
\begin{align*}
\!\!\!\!
(S_k\phi,\psi)_{L^2(\Gamma)}
=  \int_{\R^{n-1}} S_k^\infty\phi(\tbx)\overline{\psi(\tbx)}\, d\tbx
= \int_{\R^{n-1}} \widehat{S_k^\infty\phi}(\bxi)\overline{\widehat{\psi}(\bxi)}\, d\bxi
= \frac{\ri}{2 }\int_{\R^{n-1}} \frac{1}{Z(\bxi)} \widehat{\phi}(\bxi) \overline{\widehat{\psi}(\bxi)} \,\rd \bxi.
\end{align*}
\end{proof}
\section{$k$-explicit analysis of $S_k$}
\label{sec:SkAnalysis}
Our $k$-explicit analysis of the single-layer operator $S_k$ makes use of the following lemma.
\begin{lem}
\label{LemPhiFT}
Given $L>0$ let
\begin{align}
\label{PhiADef}
\Phi_L(\tbx ,x_n) : =
\begin{cases}
\Phi_c(\tbx,x_n), & |\tbx |\leq L,\\
0, & |\tbx |>L.
\end{cases}
\end{align}
Then there exists a constant $C>0$, independent of $k$, $L$, $\bxi$ and $x_n$, such that, for all $k>0$, $\bxi\in\R^{n-1}$, and $x_n\in \R$,
\begin{align}
\label{EquivBounded}
|\widehat{ \Phi_L}(\bxi,x_n)|\sqrt{k^2+|\bxi|^2} \leq
\begin{cases}
C(1+ (kL)^{1/2}), & n=3,\\
C \log{(2+(kL)^{-1})}(1+ (kL)^{1/2} + (k|x_2|)^{1/2}\log{(2+kL)}), & n=2.\\
\end{cases}
\end{align}

\end{lem}
\begin{proof}
For ease of presentation we present the proof only for the case $L=1$, but a simple rescaling step deals with the general case. It is convenient to introduce the notation $\txi :=|\bxi|$, and by $C>0$ we denote an arbitrary constant, independent of $k$, $\bxi$, and $x_n$, which may change from occurrence to occurrence. To prove \rf{EquivBounded} we proceed by estimating $|\widehat{ \Phi_1}(\bxi,x_n)|$ directly, using the formula \rf{FTRadial}. We treat the cases $n=3$ and $n=2$ separately. We will make use of the following well-known properties of the Bessel functions (cf.\ \cite[Sections 10.6, 10.14, 10.17]{DLMF}), where $\cB_n$ represents either $J_n$ or $H_n^{(1)}$:
\begin{align}
\label{JnEst}
|J_n(z)|\leq 1,& & & n\in\N,\,\,z>0,\\
\label{H0Est}
|H_0^{(1)}(z)|\leq C(1+|\log{z}|),& & & 0<z\leq 1\\
\label{H1Est}
|H_1^{(1)}(z)|\leq Cz^{-1},& & & 0<z\leq 1\\
\label{BnEst}
|\cB_n(z)|\leq Cz^{-1/2},& & & n\in\N,\,\,z>1,\\
\label{B0Der}
\cB_0'(z)= - \cB_1(z), & & & z>0,\\
\label{B1Der}
(z\cB_1)'(z)= z\cB_0(z), & & & z>0.
\end{align}

(i) In the case $n=3$, $|\widehat{ \Phi_1}(\bxi,x_3)|\leq |I|/(4\pi)$, %
where
\begin{align*}
\label{}
I:=\int_0^{1} \frac{\re^{\ri \tk \sqrt{r^2+x_3^2}}}{\sqrt{r^2+x_3^2}}J_0(\txi r)\,r\,\rd r = \underbrace{\int_0^{R}  \frac{\re^{\ri \tk \sqrt{r^2+x_3^2}}}{\sqrt{r^2+x_3^2}}J_0(\txi r)\,r\,\rd r}_{:=I_1} + \underbrace{\int_R^{1} \frac{\re^{\ri \tk \sqrt{r^2+x_3^2}}}{\sqrt{r^2+x_3^2}}J_0(\txi r)\,r\,\rd r}_{:=I_2},
\end{align*}
with $R:=\min\{1,1/\txi \}$ and $I_2:=0$ if $\txi\leq 1$. We distinguish three distinct subcases:
\begin{itemize}
\item When both $\tk\leq 1$ and $\txi \leq1 $, $R=1$ and \rf{JnEst} 
gives that $|I_1|\leq 1$. Also, in this case we have $I_2=0$ and $\sqrt{k^2+\xi^2}\leq \sqrt{2}$, so that
\begin{align}
\label{PEst1}
|\widehat{ \Phi_1}(\bxi,x_3)|\sqrt{k^2+\xi^2}\leq C, \qquad \tk\leq 1, \,\,\txi \leq1.
\end{align}
\item
When $\tk> 1$ and $\txi \leq \tk $, integration by parts, using the relation \rf{B0Der}, gives%
\begin{align*}
\label{}
I_1 = \frac{1}{\tk}\left[-\ri \re^{\ri \tk \sqrt{r^2+x_3^2}}J_0(\txi r)\right]_0^{R} - \frac{\ri \txi}{\tk}\int_{0}^R \re^{\ri \tk \sqrt{r^2+x_3^2}}J_1(\txi r)  \,\rd r,
\end{align*}
and by \rf{JnEst} we get
$|I_1| \leq (C/k) \left( 1 + R \xi \right) \leq C/k$.
If $\txi\leq 1$ then $R=1$ and $I_2=0$. If $\txi> 1$ then $R=1/\txi$ and
$I_2\neq 0$, and a similar integration by parts to that used for $I_1$ above,
combined with the bound \rf{BnEst}, gives 
the estimate
\footnote{In the case $\txi=\tk$ one can check (e.g.\ using Mathematica) that $I=e^{\ri\tk}(J_0(\tk)-\ri J_1(\tk)) \sim c\tk^{-1/2}$ as $\tk\to\infty$. So it would seem that the result \rf{I2EstComment} is sharp, and the $\tk^{1/2}$ inside the parentheses cannot be removed (e.g. by a further integration by parts), at least not in order to obtain a bound which is uniform in $\xi$, which is what we want here.}
\begin{align}
\label{I2EstComment}
|I_2| \leq \frac{C}{\tk}\left(1 + \txi^{1/2}\int_{1/\txi}^1 \frac{\rd r}{\sqrt{r}} \right) \leq \frac{C}{\tk}(1+\tk^{1/2}).
\end{align}
Finally, since $\sqrt{k^2+\xi^2}\leq \sqrt{2}k$ in this case, we conclude that
\begin{align}
\label{PEst2}
|\widehat{ \Phi_1}(\bxi,x_3)|\sqrt{k^2+\xi^2}\leq C(1+ k^{1/2}), \qquad \tk> 1, \,\,\txi\leq\tk.
\end{align}
\item When $\txi> 1$ and $\tk < \txi $,
$R=1/\xi$ and 
$|I_1|\leq 1/\xi$ by $\rf{JnEst}$.
Integration by parts, using the relation \rf{B1Der}, %
gives
\begin{align*}
\label{}
I_2 = \frac{1}{\txi}\left[\frac{r \re^{\ri \tk \sqrt{r^2+x_3^2}}}{\sqrt{r^2+x_3^2}}J_1(\txi r)\right]^1_{1/\txi} -\frac{1}{\txi} \int_{1/\txi}^1 r^2 \re^{\ri \tk \sqrt{r^2+x_3^2}}\left(\frac{\ri k}{r^2+x_3^2}-\frac{1}{(r^2+x_3^2)^{3/2}}\right) J_1(\txi r) \,\rd r,
\end{align*}
and using \rf{BnEst} we have
\begin{align*}
\label{}
|I_2| \leq \frac{C}{\txi}\left(1 + \frac{1 }{\txi^{1/2}} \int_{1/\txi}^1 \left(\frac{k }{\sqrt{r}} + \frac{1}{r^{3/2}}\right) \,\rd r \right)\leq \frac{C}{\txi}(1+\tk^{1/2}).
\end{align*}
Then, since $\sqrt{k^2+\xi^2}\leq \sqrt{2}\xi$ in this case, we conclude that
\begin{align}
\label{PEst3}
|\widehat{ \Phi_1}(\bxi,x_3)|\sqrt{k^2+\xi^2}\leq C(1+ k^{1/2}), \qquad \txi> 1, \,\,\tk <\txi.
\end{align}
\end{itemize}
Combining \rf{PEst1}, \rf{PEst2} and \rf{PEst3} gives the result \rf{EquivBounded} in the case $n=3$.

(ii) In the case $n=2$, $|\widehat{ \Phi_1}(\bxi,x_2)| \leq |I|/(2\sqrt{2\pi})$, %
where now
\begin{align*}
\label{}
I:=\int_0^{1} H_0^{(1)}(\tk \sqrt{r^2+x_2^2})\cos(\txi r)\,\rd r = \underbrace{\int_0^{R} H_0^{(1)}(\tk \sqrt{r^2+x_2^2})\cos(\txi r)\,\rd r}_{:=I_1} + \underbrace{\int_R^{1} H_0^{(1)}(\tk\sqrt{r^2+x_2^2})\cos(\txi r)\,\rd r}_{:=I_2},
\end{align*}
with $R=\min\{1,1/\tk \}$ and $I_2:=0$ for $\tk\leq 1$.
We distinguish the same three subcases as before:
\begin{itemize}
\item
When both $\tk\leq 1$ and $\txi \leq1 $, $R=1$, and from the monotonicity of $|H_0^{(1)}(z)|$, combined with the bound \rf{H0Est},
we get
\begin{align*}
\label{}
|I_1|\leq \int_0^{1} |H_0^{(1)}(\tk r)|\,\rd r \leq C(1+|\log{\tk}|) \leq C\log{(2+k^{-1})}.
\end{align*}
In this case we also have $I_2=0$ and $\sqrt{k^2+\xi^2}\leq \sqrt{2}$, so that
\begin{align}
\label{PEst1a}
|\widehat{ \Phi_1}(\bxi,x_2)|\sqrt{k^2+\xi^2}\leq C\log{(2+k^{-1})}, \qquad \tk\leq 1, \,\,\txi \leq1.
\end{align}
\item When $\tk> 1$ and $\txi \leq \tk $, $R=1/\tk$ and 
$|I_1|\leq C/k$ after use of \rf{H0Est}.
Integration by parts, using the relation \rf{B1Der}, gives
\begin{align*}
\label{}
I_2 = \frac{1}{\tk}\left[\frac{\sqrt{r^2+x_2^2}}{r}H_1^{(1)}(\tk \sqrt{r^2+x_2^2}) \cos{\txi r}\right]^1_{1/\tk} + \frac{1}{\tk} \int_{1/\tk}^1 \frac{\sqrt{r^2+x_2^2}}{r}H_1^{(1)}(\tk \sqrt{r^2+x_2^2}) \left( \txi \sin \txi r + \frac{\cos{\txi r}}{ r}\right) \,\rd r,
\end{align*}
and using \rf{BnEst} we have
\begin{align*}
\label{}
|I_2| &\leq \frac{C}{k} \left( 1 + \sqrt{k|x_2|}+ \frac{1}{k^{1/2}}\int_{1/\tk}^1\left(\frac{k}{\sqrt{r}} + \frac{1}{r^{3/2}} + \frac{k|x_2|^{1/2}}{r} + \frac{|x_2|^{1/2}}{r^2} \right)\,\rd r \right)\\
& \leq \frac{C}{\tk}\left(1+\tk^{1/2} + \sqrt{k|x_2|}\log{(2+k)}\right).
\end{align*}
Then, since $\sqrt{k^2+\xi^2}\leq \sqrt{2}k$ in this case, we have
\begin{align}
\label{PEst2a}
|\widehat{ \Phi_1}(\bxi,x_2)|\sqrt{k^2+\xi^2}\leq C \left(1+\tk^{1/2} + \sqrt{k|x_2|}\log{(2+k)}\right), 
\qquad \tk> 1, \,\,\txi \leq\tk.
\end{align}
\item When $\txi>1$ and $\tk < \txi $, integration by parts, using the relation \rf{B0Der}, gives
\begin{align}
\label{I1Estn2}
I_1 = \frac{1}{\txi}\left[H_0^{(1)}(\tk \sqrt{r^2+x_2^2}) \sin{\txi r}\right]_0^{R} + \frac{\tk}{\txi}\int_{0}^R \frac{r}{\sqrt{r^2+x_2^2}}H_1^{(1)}(\tk \sqrt{r^2+x_2^2}) \sin \txi r \,\rd r.
\end{align}
Noting that (cf.\ \cite[\S10.8 and \S10.17]{DLMF}) there exist $c_0,c_1>0$ such that $H_1^{(1)}(z) = c_0/z + F_1(z)$ for $z>0$,
where $|F_1(z)|\leq c_1$ for $z>0$, we can rewrite the second term in \rf{I1Estn2} as
\begin{align*}
\label{}
\frac{1}{\txi}\int_{0}^R \left(\frac{c_0r}{r^2+x_2^2} + \frac{kF_1r}{\sqrt{r^2+x_2^2}}\right)\sin \txi r \,\rd r.
\end{align*}
If $\tk\leq 1$ then $R=1$ and, 
noting that $\int_0^1(r\sin{\xi r})/(r^2+x_2^2)\,\rd r$ is bounded uniformly in $\xi>1$ and $x_2\in \R$, we have
$|I_1| \leq (C/\xi)\log{(2+k^{-1})}$.
Note also that $I_2=0$ in this case. On the other hand, if $\tk> 1$ then $R=1/\tk$ and $|I_1| \leq C/\txi$.
In this case $I_2\neq 0$, and a similar integration by parts to that used for $I_1$ above, combined with the bound \rf{BnEst}, gives $|I_2|\leq  (C/\xi)(1+k^{1/2})$. 
Finally, since $\sqrt{k^2+\xi^2}\leq \sqrt{2}\xi$ in this case, we conclude that
\begin{align}
\label{PEst3a}
|\widehat{ \Phi_1}(\bxi,x_2)|\sqrt{k^2+\xi^2}&\leq C\log{(2+k^{-1})}(1+ k^{1/2}), \qquad \txi> 1, \,\,\tk<\txi.
\end{align}
\end{itemize}
Combining \rf{PEst1a}, \rf{PEst2a} and \rf{PEst3a} gives the result \rf{EquivBounded} in the case $n=2$.
\end{proof}

Using this result we can prove:
\begin{thm}
\label{ThmSNormSmooth}
For any $s\in \R$, the single-layer operator $S_k$ defines a bounded linear operator $S_k:H^{s}_{\overline\Gamma}\to H^{s+1}(\Gamma)$, and there exists a constant $C>0$, independent of $k$ and $\Gamma$, such that, for all $\phi\in H^{s}_{\overline\Gamma}$, and with $L:=\diam{\Gamma}$,
\begin{align}
\label{SEst}
\norm{S_k \phi}{H^{s+1}_k(\Gamma)} \leq
\begin{cases}
C(1+ (kL)^{1/2})\norm{\phi}{H^{s}_k(\R^{n-1})}, & n=3,\\
C \log{(2+(kL)^{-1})}(1+ (kL)^{1/2})\norm{\phi}{H^{s}_k(\R^{n-1})}, & n=2,
\end{cases}
\qquad k>0.
\end{align}
\end{thm}

\begin{proof}%
We first prove \rf{SEst} for $\tilde{H}^s(\Gamma)$. By the density of $\scrD(\Gamma)$ in $\tilde{H}^s(\Gamma)$ it suffices to prove \rf{SEst} for $\phi\in \scrD(\Gamma)$. For $\phi\in \scrD(\Gamma)$ we first note that $S_k\phi = (S_k^L \phi)|_\Gamma$, where $S_k^L:\scrD(\R^{n-1})\to \scrD(\R^{n-1})$ is the convolution operator 
defined by 
$S_k^L\varphi: = (\Phi_L(\cdot,0)\ast \varphi)$, for $\varphi\in C_0^\infty(\R^{n-1})$, 
where $\Phi_L$ is defined as in \rf{PhiADef}. 
While $S_k^\infty\varphi \in C^\infty(\R^{n-1})$ for $\varphi\in \scrD(\R^{n-1})$, the fact that $\Phi_L$ has compact support means that $S_k^L\varphi\in \scrD(\R^{n-1})\subset H^{s+1}(\R^{n-1})$ (cf.\ \cite[Corollary 5.4-2a]{Zemanian}). 
Therefore, for $\phi\in \scrD(\Gamma)$ we can estimate 
$\norm{S_k \phi}{H^{s+1}_k(\Gamma)}\leq \normt{S_k^L \phi}{H^{s+1}_k(\R^{n-1})}$, 
and since 
$\widehat{S_k^L \varphi}(\bxi) = \widehat{( \Phi_L(\cdot,0)\ast \varphi)}(\bxi) = (2\pi)^{(n-1)/2}\widehat{ \Phi_L}(\bxi,0)\hat{\varphi}(\bxi)$ 
for any $\varphi\in \scrD(\R^{n-1})$, 
the bound \rf{SEst} follows from Lemma \ref{LemPhiFT}. 

Finally, we can extend the bound \rf{SEst} to $\phi\in H^{s}_{\overline\Gamma}$ (without changing the constant) by appealing to Remark \ref{rem:ss_equality} (we can approximate $\phi\in H^{s}_{\overline\Gamma}$ arbitrarily well by an element of $\tilde H^s (\Gamma_\eps)$ for $\Gamma_\eps$ an arbitrarily small neighbourhood of $\Gamma$).
\end{proof}

\begin{thm}
\label{ThmCoerc}
The sesquilinear form on $\tilde{H}^{-1/2}(\Gamma)\times \tilde{H}^{-1/2}(\Gamma)$ defined by
\begin{align*}
\label{}
a(\phi,\psi):=\langle S_k\phi,\psi \rangle_{H^{1/2}(\Gamma)\times \tilde{H}^{-1/2}(\Gamma)}, \quad \phi,\psi\in \tilde{H}^{-1/2}(\Gamma),
\end{align*}
satisfies the coercivity estimate
\begin{align}
\label{SCoercive}
|a(\phi,\phi)|
\geq \frac{1}{2\sqrt{2}} \norm{\phi}{\tilde{H}^{-1/2}_k(\Gamma)}^2,\quad \phi\in \tilde{H}^{-1/2}(\Gamma), \,\, k>0.
\end{align}
\end{thm}

\begin{proof}%
By the density of $\scrD(\Gamma)$ in $\tilde{H}^{-1/2}(\Gamma)$ it suffices to prove \rf{SCoercive} for $\phi\in \scrD(\Gamma)$.  For such a $\phi$, formula \rf{SIPFourierRep} from Theorem \ref{ThmFourierRep} gives
\begin{align}
\label{aCoercProof}
 |a(\phi,\phi)|  = \frac{1}{2}\left|\int_{\R^{n-1}} \frac{|\widehat{\phi}(\bxi)|^2}{Z(\bxi)} \,\rd \bxi \right|
\geq %
\frac{1}{2\sqrt{2}} \int_{\R^{n-1}} \frac{|\widehat{\phi}(\bxi)|^2}{\sqrt{|k^2-|\bxi|^2|}} \,\rd \bxi
\geq\frac{1}{2\sqrt{2}} \int_{\R^{n-1}} \frac{|\widehat{\phi}(\bxi)|^2}{\sqrt{k^2+|\bxi|^2}} \,\rd \bxi,
\end{align}
as claimed.
\end{proof}

\begin{rem}
We can show that the bounds established in Theorem \ref{ThmSNormSmooth} are sharp in their dependence on $k$ as $k\to\infty$, at least for the case $s=-1/2$. For simplicity of presentation we assume that $\diam{\Gamma}=1$ and $k>1$. Let $\phi(\tbx):=\re^{\ri k\tbd\cdot \tbx}\psi(\tbx)$ for $\tbx\in\R^{n-1}$, where $\tbd\in \R^{n-1}$ is a unit vector and $0\neq \psi\in \scrD(\Gamma)$ is independent of $k$, depending only on the shape of $\Gamma$. Then $\widehat{\phi}(\bxi) = \hat{\psi}(\boldsymbol{\eta})$, where $\boldsymbol{\eta}=\bxi-k\tbd$, and for any $\eta_*\geq 1$ the first inequality in \rf{aCoercProof} gives that
\begin{align}
\label{CRem1}
|a(\phi,\phi)| \geq %
\frac{1}{2\sqrt{2} k} \int_{\R^{n-1}} \frac{|\widehat{\psi}(\boldsymbol{\eta})|^2}{\sqrt{|1-|\tbd+\boldsymbol{\eta}/k|^2|}} \,\rd \boldsymbol{\eta}
\geq \frac{1}{2\sqrt{6}\eta_*k^{1/2}} \int_{|\boldsymbol{\eta}|\leq \eta_*} |\widehat{\psi}(\boldsymbol{\eta})|^2\,\rd \boldsymbol{\eta},
\end{align}
since $|1-|\tbd+\boldsymbol{\eta}/k|^2| \leq (1/k)|2\tbd\cdot \boldsymbol{\eta} + |\boldsymbol{\eta}|^2/k |  \leq3\eta_*^2/k$, 
for $|\boldsymbol{\eta}|\leq \eta_*$. %
Also, for the same choice of $\phi$,
\begin{align}
\label{CRem2}
 \norm{\phi}{\tilde{H}^{-1/2}_k(\Gamma)}^2 = \frac{1}{k} \int_{\R^{n-1}} \frac{|\widehat{\psi}(\boldsymbol{\eta})|^2}{\sqrt{|1+|\tbd+\boldsymbol{\eta}/k|^2|}} \,\rd \boldsymbol{\eta}
\leq  \frac{1}{k} \int_{\R^{n-1}} |\widehat{\psi}(\boldsymbol{\eta})|^2\,\rd \boldsymbol{\eta}
\leq  \frac{2}{k} \int_{|\boldsymbol{\eta}|\leq \eta_*} |\widehat{\psi}(\boldsymbol{\eta})|^2\,\rd \boldsymbol{\eta},
\end{align}
for $\eta_*$ sufficiently large.
Combining \rf{CRem1} and \rf{CRem2} gives 
$|a(\phi,\phi)| \geq (k^{1/2}/(4\sqrt{6}\eta_*))\normt{\phi}{\tilde{H}^{-1/2}_k(\Gamma)}^2,$
and then, since $|a(\phi,\phi)| \leq  \norm{S_k\phi}{H^{1/2}_k(\Gamma)} \norm{\phi}{\tilde{H}^{-1/2}_k(\Gamma)}$, we conclude that, for this particular choice of $\phi$,%
\begin{align*}
\label{}
 \norm{S_k\phi}{H^{1/2}_k(\Gamma)}  \geq (k^{1/2}/(4\sqrt{6}\eta_*)) \norm{\phi}{\tilde{H}^{-1/2}_k(\Gamma)},
\end{align*}
which demonstrates the sharpness of \rf{SEst} in the limit $k\to\infty$.
\end{rem}

\section{$k$-explicit analysis of $T_k$}
\label{sec:TkAnalysis}
This section improves upon and generalizes the result of \cite[Theorem 2]{Ha-Du:92}, sharpening the $k$-dependence of the bounds on the coercivity constant, and providing estimates in the 2D case, which was not considered in \cite{Ha-Du:92}.

\begin{thm}
\label{ThmTNormSmooth}
For any $s\in\R$, the hypersingular operator $T_k$ defines a bounded linear operator $T_k:H^{s}_{\overline\Gamma}\to H^{s-1}(\Gamma)$, and
\begin{align}
\label{TEst}
\norm{T_k \phi}{H^{s-1}_k(\Gamma)} \leq \frac{1}{2}\norm{\phi}{H^{s}_k(\R^{n-1})}, \quad \phi\in H^{s}_{\overline\Gamma}, \,\,k>0.
\end{align}
\end{thm}

\begin{proof}
Again, we give the proof for $\phi\in\tilde{H}^s(\Gamma)$, the extension to $\phi\in H^{s}_{\overline\Gamma}$ being justified using Remark \ref{rem:ss_equality}.  
By the density of $\scrD(\Gamma)$ in $\tilde{H}^s(\Gamma)$ it suffices to prove \rf{TEst} for $\phi\in \scrD(\Gamma)$. For such a $\phi$ we first note from Theorem \ref{ThmFourierRep} that $T_k\phi = (T_k^\infty \phi)|_\Gamma$, where 
$\widehat{T_k^\infty\varphi}(\bxi) = (\ri/2) Z(\bxi) \hat{\varphi}(\bxi)$, for $\varphi\in \scrD(\R^{n-1})$.
Clearly, for any $\varphi\in \scrD(\R^{n-1})$ and any $s\in \R$, the integral
\begin{align*}
\label{}
\int_{\R^{n-1}}(k^2+|\bxi|^2)^{s-1} |\widehat{T_k^\infty \varphi}(\bxi)|^2 \, \rd \bxi
\end{align*}
is finite, and hence $T_k^\infty\varphi\in H^{s-1}(\R^{n-1})$. As a result, given $\phi\in \scrD(\Gamma)$ we can estimate
\begin{align}
\label{}
\norm{T_k \phi}{H^{s-1}_k(\Gamma)} \leq \norm{T_k^\infty \phi}{H^{s-1}_k(\R^{n-1})}
& = \frac{1}{2}\sqrt{\int_{\R^{n-1}}(k^2+|\bxi|^2)^{s-1} |Z(\bxi)|^2|\widehat{\phi}(\bxi)|^2 \, \rd \bxi}
\notag\\ &
 \leq \frac{1}{2}\sqrt{\int_{\R^{n-1}}(k^2+|\bxi|^2)^{s}| \widehat{\phi}(\bxi)|^2 \, \rd \bxi},%
\end{align}
as required.
\end{proof}

\begin{thm}
\label{ThmTCoercive}
The sesquilinear form on $\tilde{H}^{1/2}(\Gamma)\times \tilde{H}^{1/2}(\Gamma)$ defined by
\begin{align*}
\label{}
b(\phi,\psi):=\langle T_k\phi,\psi \rangle_{H^{-1/2}(\Gamma)\times \tilde{H}^{1/2}(\Gamma)}, \quad \phi,\psi\in \tilde{H}^{1/2}(\Gamma),
\end{align*}
satisfies, for any $k_0>0$, the coercivity estimate
\begin{align}
\label{bEst}
|b(\phi,\phi)|\geq 
C(kL)^{\beta}
\norm{\phi}{\tilde{H}^{1/2}_k(\Gamma)}^2, \quad \phi\in \tilde{H}^{1/2}(\Gamma),\,\,k\geq k_0,
\end{align}
where $L:=\diam{\Gamma}$, 
$C>0$ is a constant depending only on $k_0$, %
and
\begin{align*}
\label{}
\beta=
\begin{cases}
-\frac{2}{3},&n=3,\\
-\frac{1}{2}, & n=2.
\end{cases}
\end{align*}
\end{thm}

\begin{proof}
As in the proof of Lemma \ref{LemPhiFT}, for ease of presentation we present the proof only for the case $L=1$, but a simple rescaling deals with the general case. 
By the density of $\scrD(\Gamma)$ in $\tilde{H}^{1/2}(\Gamma)$ it suffices to prove \rf{bEst} for $\phi\in \scrD(\Gamma)$. For such a $\phi$, equation \rf{TIPFourierRep} from Theorem \ref{ThmFourierRep} gives that
\begin{align}
\label{bI}
|b(\phi,\phi)| = \frac{1}{2}\left|\int_{\R^{n-1}} Z(\bxi) |\widehat{\phi}(\bxi)|^2 \,\rd \bxi\right| \geq \frac{I}{2\sqrt{2}},
\end{align}
where
\begin{align*}
I:=\int_{\R^{n-1}} |Z(\bxi)| |\widehat{\phi}(\bxi)|^2 \,\rd \bxi.
\end{align*}
Defining
\begin{align*}
\label{}
J:=\norm{\phi}{\tilde{H}^{1/2}_k(\Gamma)}^2 =\int_{\R^{n-1}}(k^2+|\bxi|^2)^{1/2}|\widehat{\phi}(\bxi)|^2\,\rd \bxi,
\end{align*}
the problem of proving \rf{bEst} reduces to that of proving
\begin{align}
\label{JIIneq}
I \geq Ck^\beta J, \qquad k\geq k_0,
\end{align}
for some $C$ depending only on $k_0$. %
However, proving \rf{JIIneq} is complicated by the fact that the factor $|Z(\bxi)|$ in $I$ vanishes when $|\bxi|=k$. To deal with this, we split the integrals $I$ and $J$ into
\begin{align*}
\label{}
I&= I_1+I_2+I_3+I_4, \qquad\qquad
J= J_1+J_2+J_3+J_4,
\end{align*}
using the decomposition
\begin{align*}
\label{}
\int_{\R^{n-1}} = \int_{0<|\bxi|< k-\eps} + \int_{k-\eps<|\bxi|< k} + \int_{k<|\bxi|<k+\eps} + \int_{|\bxi|>k+\eps}
\end{align*}
(where $0<\eps\leq k$ is to be specified later), and proceed to estimate the integrals $J_1$ to $J_4$ separately.

We first estimate
\begin{align*}
J_1 :&= \int_{0<|\bxi|<k-\eps} (k^2+|\bxi|^2)^{1/2}|\widehat{\phi}(\bxi)|^2\,\rd \bxi
= \int_{0<|\bxi|<k-\eps} \frac{(k^2+|\bxi|^2)^{1/2}}{ (k^2-|\bxi|^2)^{1/2}} (k^2-|\bxi|^2)^{1/2}|\widehat{\phi}(\bxi)|^2\,\rd \bxi \leq C_1 I_1,
\end{align*}
where
\begin{align}
\label{C1Def}
C_1:= \frac{\sqrt{2}k}{(\eps(2k-\eps))^{1/2}},
\end{align}
and we have used the fact that, for $0<|\bxi|<k-\eps$,
\begin{align*}
\label{}
\frac{k^2+|\bxi|^2}{k^2-|\bxi|^2}
\leq\frac{k^2+(k-\eps)^2}{k^2-(k-\eps)^2}
\leq C_1^2.
\end{align*}
Similarly,
\begin{align*}
J_4 :&= \int_{|\bxi|>k+\eps}(k^2+|\bxi|^2)^{1/2}|\widehat{\phi}(\bxi)|^2\,\rd \bxi
= \int_{|\bxi|>k+\eps} \frac{(k^2+|\bxi|^2)^{1/2}}{ (|\bxi|^2-k^2)^{1/2}} (|\bxi|^2-k^2)^{1/2}|\widehat{\phi}(\bxi)|^2\,\rd \bxi \leq C_2 I_4,
\end{align*}
where
\begin{align}
\label{C2Def}
C_2:= \frac{(k^2+(k+\eps)^2)^{1/2}}{(\eps(2k+\eps))^{1/2}},
\end{align}
and we have used the fact that, for $|\bxi|>k+\eps$,
\begin{align*}
\label{}
\frac{k^2+|\bxi|^2}{|\bxi|^2-k^2}
= \frac{\frac{k^2}{|\bxi|^2}+1}{1-\frac{k^2}{|\bxi|^2}}
\leq\frac{\frac{k^2}{(k+\eps)^2}+1}{1-\frac{k^2}{(k+\eps)^2}}
= \frac{k^2+(k+\eps)^2}{(k+\eps)^2-k^2}=C_2^2.
\end{align*}
To estimate $J_2$ and $J_3$, we first derive a pointwise estimate on the Fourier transform of $\phi$. To do this, we note that, for $\bxi_1,\bxi_2\in \R^{n-1}$,
\begin{align*}
\label{}
|\widehat{\phi}(\bxi_1)-\widehat{\phi}(\bxi_2)| &\leq \frac{1}{(2\pi)^{\frac{n-1}{2}}}\left|\int_{ \Gamma}\re^{-\ri \bxi_2\cdot \bx }\left(\re^{-\ri (\bxi_1-\bxi_2)\cdot \bx  }- 1 \right)\phi(\bx )\,\rd \bx  \right|\nonumber \\
&\leq \frac{|\bxi_1-\bxi_2|}{(2\pi)^{\frac{n-1}{2}}}\int_{ \Gamma}|\bx ||\phi(\bx )|\,\rd \bx  \nonumber \\
&\leq \frac{|\bxi_1-\bxi_2|}{(2\pi)^{\frac{n-1}{2}}}\left(\int_{ \Gamma}|\bx |^2\,\rd \bx \right)^{1/2} \left(\int_{ \Gamma}|\phi(\bx )|^2\,\rd \bx \right)^{1/2}.  \nonumber \\
&=\tilde{c}_n |\bxi_1-\bxi_2|\left(\int_{\R^{n-1}}|\widehat{\phi}(\bxi)|^2\,\rd \bxi\right)^{1/2},  \nonumber \\
&\leq \tilde{c}_n|\bxi_1-\bxi_2|k^{-1/2}J,
\end{align*}
where $\tilde{c}_n>0$ represents a constant depending only on $n$, whose value may change from occurrence to occurrence. Here we have used our assumption that $\diam{\Gamma}=1$, and the fact that
\begin{align*}
\label{}
|\re^{\ri t}-1|^2 = (\cos{t}-1)^2 + \sin^2{t} = 2(1-\cos{t}) = 4\sin^2{(t/2)}\leq t^2, \quad t\in\R.
\end{align*}
As a result, we can estimate, with $\hat{\bxi}:=\bxi/|\bxi|$,
\begin{align}
\label{FTPointwise}
|\widehat{\phi}(\bxi)|^2 \leq 2\left (|\widehat{\phi}(\bxi\pm\eps \hat{\bxi})|^2 + |\widehat{\phi}(\bxi\pm\eps \hat{\bxi}) -
\widehat{\phi}(\bxi)|^2\right) \leq 2\left (|\widehat{\phi}(\bxi\pm\eps \hat{\bxi})|^2 + \frac{\eps^2 \tilde{c}_n J}{k}\right), %
\end{align}
which %
then implies that
\begin{align}
\label{J2Est}
J_2:&= \int_{k-\eps<|\bxi|<k} (k^2+|\bxi|^2)^{1/2}|\widehat{\phi}(\bxi)|^2\,\rd \bxi %
 \leq 2\sqrt{2}k\left(\int_{k-\eps<|\bxi|<k} |\widehat{\phi}(\bxi-\eps \hat{\bxi})|^2 \,\rd \bxi + \frac{\eps^3 a_n\tilde{c}_n J}{k} \right),
\end{align}
where $a_n:=k-\eps/2$ if $n=3$ and $a_n:=1$ if $n=2$.
We now note that, for any $0<\eps<c<d$,
\begin{align}
\label{fxipluseps}
\int_{c<|\bxi|<d} f(\bxi \pm \eps \hat{\bxi}) \,\rd \bxi = \begin{cases}
\int_{c\pm \eps<|\bxi|<d\pm \eps} f(\bxi)\left(1\mp \frac{\eps}{|\bxi|}\right) \,\rd \bxi,& n=3,\\
\int_{c\pm \eps<|\bxi|<d\pm \eps} f(\bxi) \,\rd \bxi,& n=2.
\end{cases}
\end{align}
Assume $0<\eps<k/3$. Then for $k-2\eps<|\bxi|<k-\eps$, we have $1+\eps/|\bxi|\leq 2$, so that, using \rf{fxipluseps},
\begin{align*}
\label{}
\int_{k-\eps<|\bxi|<k} |\widehat{\phi}(\bxi-\eps \hat{\bxi})|^2 \,\rd \bxi &\leq 2 \int_{k-2\eps<|\bxi|<k-\eps} |\widehat{\phi}(\bxi )|^2 \,\rd \bxi\nonumber\\
&\leq \frac{2}{(k^2-(k-\eps)^2)^{1/2}}\int_{k-2\eps<|\bxi|<k-\eps} (k^2-|\bxi|^2)^{1/2}|\widehat{\phi}(\bxi )|^2 \,\rd \bxi\nonumber\\
&\leq \frac{2}{(\eps(2k-\eps))^{1/2}}I_1,
\end{align*}
and, inserting this estimate into \rf{J2Est}, we find that
\begin{align}
\label{J2Est2}
J_2 \leq 4C_1 I_1 + \eps^3 a_n \tilde{c}_n J,
\end{align}
where $C_1$ is defined as in \rf{C1Def}.

Applying a similar procedure, but using \rf{FTPointwise} and \rf{fxipluseps} with the plus rather than the minus sign, and again assuming that $0<\eps<k/3$, gives
\footnote{The factor of $2$ in front of $C_2I_4$ here, rather than $4$, as in \rf{J2Est}, is not a typographical error - it arises because of the change of sign in \rf{fxipluseps}.}
\begin{align*}
\label{}
J_3 :&= \int_{k<|\bxi|<k+\eps} (k^2+|\bxi|^2)^{1/2}|\widehat{\phi}(\bxi)|^2\,\rd \bxi\leq 2C_2 I_4 + \eps^3 b_n \tilde{c}_n J,
\end{align*}
where $C_2$ is defined as in \rf{C2Def} and
$b_n:=k+\eps/2$ if $n=3$ and $b_n:=1$ if $n=2$.

Since $a_n+b_n=2k^{n-2}$,
we can summarize by saying that, 
for $0<\eps<k/3$,%
\begin{align*}
J = J_1+J_2+J_3+J_4 & \leq 5C_1I_1 + 3C_2I_4 +  \eps^3 k^{n-2}\tilde{c}_n J,
\end{align*}
which implies that
\begin{align}
\label{JEst2}
J(1-C_4) \leq C_3I,
\end{align}
where
\begin{align*}
\label{}
C_3&:=\max\{5C_1,3C_2\},\qquad \qquad
C_4:=\eps^3 k^{n-2}\tilde{c}_n,
\end{align*}
for some constant $\tilde{c}_n>0$ depending only on $n$. 
Now, if we choose $\eps$ small enough to make $C_4\leq 1/2$, then \rf{JEst2} will imply \rf{JIIneq}, once we have determined the $k$-dependence of $C_3$. 
Explicitly, 
setting $\eps=c_\eps k^{p_\eps}$, 
we can ensure that $C_4\leq 1/2$, uniformly in $k$, by choosing
$p_\eps=(-1/3)^{n-2}$
and taking $c_\eps$ sufficiently small. Furthermore, in order to satisfy the requirement that $0<\eps<k/3$, we restrict our attention to $k\geq k_0>0$ and choose $c_\eps$ depending on $k_0$. 
We then have that%
\begin{align*}
\label{}
C_3\leq \frac{\tilde{c}_n}{\sqrt{c_\eps}} k^{1/2-p_{\eps}/2},%
\end{align*}
which, recalling \rf{JEst2}, gives \rf{JIIneq} (and hence \rf{bEst}) with $\beta = -1/2+p_{\eps}/2$.
\end{proof}

\begin{rem}
\label{TRem}
We can show that the bounds established in Theorem \ref{ThmTNormSmooth} are sharp in their dependence on $k$ as $k\to\infty$, at least for the case $s=1/2$. As before, for simplicity of presentation we assume that $\diam{\Gamma}=1$ and $k>1$. Let $0\neq \phi\in \scrD(\Gamma)$ be independent of $k$. Then, by \rf{bI},%
\begin{align}
\label{bEstRem}
|b(\phi,\phi)| \geq \frac{1}{2\sqrt{2}} \int_{\R^{n-1}}\sqrt{|k^2-|\bxi|^2|}|\widehat{\phi}(\bxi)|^2 \,\rd \bxi
& \geq  \frac{\sqrt{3}k}{4\sqrt{2}} \int_{|\bxi|\leq k/2}\,|\widehat{\phi}(\bxi)|^2 \,\rd \bxi.
\end{align}
Also, 
\begin{align}
\label{bEstRem2}
 \norm{\phi}{\tilde{H}^{1/2}_k(\Gamma)}^2 %
\leq  \frac{3k}{2} \int_{|\bxi|\leq k/2} |\widehat{\phi}(\bxi)|^2\,\rd \bxi + 3\int_{|\bxi|> k/2} |\bxi||\widehat{\phi}(\bxi)|^2\,\rd \bxi
\leq  2k \int_{|\bxi|\leq k/2} |\widehat{\phi}(\bxi)|^2\,\rd \bxi ,
\end{align}
for $k$ sufficiently large.
Combining \rf{bEstRem} and \rf{bEstRem2} gives $|b(\phi,\phi)| \geq \sqrt{3}/(8\sqrt{2})\norm{\phi}{\tilde{H}^{1/2}_k(\Gamma)}^2$, and then, since $|b(\phi,\phi)| \leq  \norm{T_k\phi}{H^{-1/2}_k(\Gamma)} \norm{\phi}{\tilde{H}^{1/2}_k(\Gamma)}$,
we conclude that%
\begin{align*}
\label{}
 \norm{T_k\phi}{H^{-1/2}_k(\Gamma)}  \geq  \sqrt{3}/(8\sqrt{2})  \norm{\phi}{\tilde{H}^{1/2}_k(\Gamma)},
\end{align*}
for $k$ sufficiently large, which demonstrates the sharpness of \rf{TEst} in the limit $k\to\infty$.
\end{rem}

\section{Norm estimates in $H^{1/2}(\Gamma)$}
\label{sec:NormEstimates}
In this section we derive $k$-explicit estimates of the norms of certain functions in $H^{1/2}(\Gamma)$, which are of relevance to the numerical solution of the boundary value problem $\sD''$ for scattering by a sound-soft screen, when it is solved via the single-layer integral equation formulation \rf{BIE_sl}; for an application of the results presented here see e.g.\ \cite{ScreenBEM}. Suppose that we are attempting to solve this integral equation using a Galerkin BEM, and that we have (by some means) obtained an error estimate in $\tilde{H}^{-1/2}(\Gamma)$ for the Galerkin solution. In order to derive error estimates for the resulting solution in the domain $D$, and the far-field pattern, we need to estimate integrals (strictly speaking, duality pairings - see below) of the form
\begin{align}
\label{IDef}
I:=\int_\Gamma w(\by) v(\by)\,\rd s(\by),
\end{align}
where $v\in \tilde{H}^{-1/2}(\Gamma)$ represents the error in our Galerkin solution, and $w\in H^{1/2}(\Gamma)$ is a known function, possibly depending on a parameter. For example, in the case of the far-field pattern we would have $w(\by)= \re^{\ri k\hat{\bx}\cdot \by}$ for some observation direction $\hat{\bx}\in\R^n$ with $|\hat{\bx}|= 1$. In the case of the solution in the domain, we would have $w(\by)= \Phi(\bx,\by)$ for some observation point $\bx\in D$. 
One also needs to estimate integrals of the form \rf{IDef} when attempting to estimate the magnitude of the solution of the continuous problem at a point $\bx$ in the domain, using a bound on the boundary data (cf.\ Corollary \ref{cor:BoundOnSolnInDomain}). In this case $w(\by)=\Phi(\bx,\by)$ and $v=[\pdonetext{u}{\bn}]$ is the exact solution of \rf{BIE_sl}. 

Of course the integral \rf{IDef} should be interpreted as the duality pairing $| \langle w,\overline{v} \rangle_{H^{1/2}(\Gamma)\times \tilde{H}^{-1/2}(\Gamma)}|$, and
\begin{align*}
\label{}
|I| =| \langle w,\overline{v} \rangle_{H^{1/2}(\Gamma)\times \tilde{H}^{-1/2}(\Gamma)}| \leq  \normt{w}{H^{1/2}_k(\Gamma)}\normt{v}{\tilde{H}^{-1/2}_k(\Gamma)}.
\end{align*}
We are assuming that we have an estimate of the quantity $\normt{v}{\tilde{H}^{-1/2}_k(\Gamma)}$. It therefore remains to bound $\normt{w}{H^{1/2}_k(\Gamma)}$.

\begin{lem}
\label{lem:H_one_half_norms}
Let $k>0$, let $\Gamma$ be an arbitrary nonempty relatively open subset of $\Gamma_\infty$, and let $L:=\diam{\Gamma}$. 
\begin{enumerate}[(i)]
\item Let $\bd\in\R^n$ with $||\bd||\leq 1$. Then there exists $C>0$, independent of $k$ and $\Gamma$, such that 
\begin{align}
 \label{eqn:H_one_half_planewave}
 \normt{\re^{\ri k \bd\cdot (\cdot)}}{H^{1/2}_k(\Gamma)} \leq CL^{(n-2)/2}(1+\sqrt{kL}).
 \end{align} 
\item Let $\bx\in D:=\R^n\setminus\overline{\Gamma}$. Then there exists $C>0$, independent of $k$ and $\Gamma$, such that
 \begin{align}
 \label{eqn:H_one_half_Phi}
 \norm{\Phi_k(\bx,\cdot)}{H^{1/2}_k(\Gamma)} \leq 
\begin{cases}
 Ck^{1/2}\left(1+ \frac{1}{\sqrt{kL}}\right)\left(1+ \frac{1}{(kd)^{3/2}}\right)\left(1+kL\right), & n=3,\\
 C\left(1+\frac{1}{\sqrt{kL}}\right)\left(1+\frac{1}{\sqrt{kd}}\right)\log\left(2+\frac{1}{kd}\right)\log^{1/2}(2+kL), & n=2,
\end{cases}
 \end{align} 
 where $d:=\dist(\bx,\Gamma)$.
\end{enumerate}
\end{lem}

\begin{proof}
In both parts (i) and (ii) we wish to estimate $\normt{u}{H^{1/2}(\Gamma)}$, where $u\in H^{1/2}(\Gamma)$ is such that there exists a closed set $F\subset \R^{n-1}$ containing $\tGamma$, and having the same diameter as $\Gamma$, and an open set $K$ containing a neighbourhood of $F$ (e.g. $K=F_\delta:=\{\by\in\R^{n-1}:\dist(\by,F)<\delta\}$ for some $\delta>0$), such that $u=\tilde{u}|_\Gamma$ for some $\tilde{u}\in L^1_{\rm loc}(\R^{n-1})$ with $\tilde{u}|_K\in \scrD(\overline K)$.

Let $F':=\{\by/L:\by\in F\}$ be a scaled version of $F$ with $\diam(F')=1$. Then given any $\eps>0$ there exists (e.g.\ by McLean \cite[Thm 3.6]{McLean}) a cut-off function $\chi\in \scrD(\R^{n-1})$ such that $0\leq\chi(\bx)\leq 1$ for all $\bx\in\R^{n-1}$ with $\chi(\bx)=1$ if $\bx\in F'$ and $\chi(\bx)=0$ if $\bx\in\R^{n-1}\setminus F'_{\eps}$, and $|\pdonetext{\chi}{x_j}(\bx)|\leq C/\eps$ for $j=1,\ldots,n-1$ and $\bx\in F'_{\eps} \setminus F'$. Define $\chi_L\in \scrD(\R^{n-1})$ by $\chi_L(\bx):=\chi(\bx/L)$. Then $0\leq\chi_L(\bx)\leq 1$ for all $\bx\in\R^{n-1}$ with $\chi_L(\bx)=1$ if $\bx\in F$ and $\chi_L(\bx)=0$ if $\bx\in\R^{n-1}\setminus F_{L\eps}$, and $|\pdonetext{\chi_L}{x_j}(\bx)|\leq C/(L\eps)$ for $j=1,\ldots,n-1$ and $\bx\in F_{L\eps} \setminus F$. 
Then $(\chi_L \tilde{u})|_\Gamma=u$, so that, by the definition of $\normt{\cdot}{H^{1/2}_k(\Gamma)}$,
\begin{align}
\label{uChi}
\normt{u}{H^{1/2}_k(\Gamma)} \leq \normt{\chi_L \tilde{u}}{H^{1/2}_k(\R^{n-1})}.
\end{align}
We now estimate the right hand side of \rf{uChi} in the specific cases relevant for (i) and (ii). 
Throughout the proof, $C$ will denote a positive constant, independent of both $k$ and $\Gamma$.

For part (i) we have $\tilde{u}(\by)= \re^{\ri k\bd\cdot \by}$, $|\bd|\leq 1$. 
In this case $\tilde{u}(\by)$ has no singularities and we can take $F$ to be the smallest closed interval (for $n=2$) or closed ball (for $n=3$) containing $\tGamma$. Moreover, the standard shift and scaling theorems for the Fourier transform imply that $\widehat{\chi_L \tilde{u}}(\bxi) = L^{n-1}\widehat{\chi}(L(\bxi-k\bd))$, and hence the right-hand side of \rf{uChi} can be bounded directly in terms of the $H^{1/2}_k$ norm of $\chi$, which can be assumed to be a constant independent of $k$ and $\Gamma$ (as we have no singularities in $\tilde{u}(y)$ and therefore no restriction on the choice of $\eps$ in this case). Explicitly, 
\begin{align*}
\label{}
\normt{\chi_L \tilde{u}}{H^{1/2}_k(\R^{n-1})}^2 &= L^{2(n-1)}\int_{\R^{n-1}} (k^2+|\bxi|^2)^{1/2}|\widehat{\chi}(L(\bxi-k\bd))|^2 \, \rd\bxi\notag\\
&= L^{n-1}\int_{\R^{n-1}} (k^2+|\boldeta/L+k\bd|^2)^{1/2}|\widehat{\chi}(\boldeta)|^2 \, \rd\boldeta\notag\\
&\leq L^{n-2}(1+kL)\int_{\R^{n-1}} (2+|\boldeta|)|\widehat{\chi}(\boldeta)|^2 \, \rd\boldeta\notag,
\end{align*}
after the change of variable $\boldeta=L(\bxi-k\bd)$. 
Using this in \rf{uChi} gives the desired result \rf{eqn:H_one_half_planewave}, 
where $C>0$ is a constant independent of $k$, and proportional to $\sqrt{\int_{\R^{n-1}} (2+|\boldeta|)|\widehat{\chi}(\boldeta)|^2 \, \rd\boldeta}$. Taking e.g.\ $\eps=1$ in the definition of $\chi$ we see that $C$ is also independent of $\Gamma$.

For part (ii), where $\tilde{u}(\by)= \Phi_k(\bx,\by)$, $\bx\in D$,  
such a direct approach is not possible. However, since $\chi_L\tilde{u}\in H^1(\R^{n-1})$ we can estimate $\normt{\chi_L \tilde{u}}{H^{1/2}_k(\R^{n-1})}$ using the fact that $H^{1}(\R^{n-1})$ can be continuously embedded in $H^{1/2}(\R^{n-1})$. 

Explicitly, when $n=2$ we have
\begin{align}
\label{uChiEst}
\normt{\chi_L \tilde{u}}{H^{1/2}_k(\R)} \leq k^{-1/2}\normt{\chi_L \tilde{u}}{H^{1}_k(\R)}%
& \leq k^{1/2}\normt{\chi_L \tilde{u}}{L^2(\R)}+  k^{-1/2} \left(  \normt{\chi_L' \tilde{u}}{L^2(\R)} + \normt{\chi_L \tilde{u}'}{L^2(\R)}\right),
\end{align}
where $f'$ represents the derivative of $f$. %
In this case, with $\by=(s,0)$, we have $\tilde{u}(\by)\equiv\tilde{u}(s)=(\ri/4)H_0^{(1)}(kr(s))$, where $r(s)=\sqrt{(x_1-s)^2+x_2^2}$. Then $\tilde{u}'(s)=(-\ri k/4)(s-x_1)/r(s) H_1^{(1)}(k r(s))$. To avoid the singularity at $\by=\bx$ we define $d:=\dist(\bx, \Gamma)$ and take our closed set $F:=I_\Gamma\setminus B_d$, where $I_\Gamma$ is the smallest closed interval containing $\Gamma$ and $B_d=\{s\in \R : r(s)<d \}$. We set $\eps=\min\{1,d/(2L)\}$, and note that ${\rm meas}(F_{L\eps})\leq L+2\eps$ and ${\rm meas}(F_{L\eps}\setminus F)\leq 4L\eps$. 
We now estimate each of the three terms on the right hand side of \rf{uChiEst} in turn. 
In doing so we make extensive use of the bounds on the Hankel functions $H_0^{(1)}$ and $H_1^{(1)}$ presented in \rf{H0Est}-\rf{BnEst}.
\begin{itemize}
\item
For the term involving $\normt{\chi_L \tilde{u}}{L^2(\R)}$, we have, using the fact that $|H_0^{(1)}(z)|$ is a decreasing function of $z$, and arguing as in the proof of \cite[Lemma 4.1]{DHSLMM},
\begin{align*}
\label{}
\normt{\chi_L \tilde{u}}{L^2(\R)}^2 \leq \frac{1}{16}\int_{F_{L\eps}} |H_0^{(1)}(k r(s))|^2 \,\rd s 
\leq \frac{1}{8}\int_0^{L(1/2+\eps)} |H_0^{(1)}(ks)|^2 \,\rd s 
\leq \frac{C}{k}\log{(2+kL(1/2+\eps))},
\end{align*}
and since $\eps\leq 1$ we get
\begin{align}
\label{n2Est1}
k^{1/2}\normt{\chi_L \tilde{u}}{L^2(\R)}\leq C\log^{1/2}{(2+kL)}.
\end{align}

\item
For the term involving $\normt{\chi_L' \tilde{u}}{L^2(\R)}$ we estimate
\begin{align*}
\label{}
\normt{\chi_L' \tilde{u}}{L^2(\R)}^2 
&= \int_{F_{L\eps}\setminus F} |\chi_L' \tilde{u}|^2
\leq \frac{C}{L\eps}\sup_{s\in F_{L\eps}\setminus F} |H_0^{(1)}(kr(s))|^2
\leq \frac{C}{L\eps}|H_0^{(1)}(kd/2)|^2,
\end{align*}
since the definition of $\eps=\min\{1,d/(2L)\}$ ensures that $\eps\leq d/(2L)$, which in turn ensures that $r(s)\geq d/2$ for $s\in F_{L\eps}$. We can estimate $|H_0^{(1)}(kd/2)|\leq C\log{(2+1/(kd))}$, and also
\begin{align}
\label{FracEst}
\frac{1}{kL\eps}\leq 2\left(\frac{1}{kL} + \frac{1}{kd}\right) \leq 2\left(1+\frac{1}{kL}\right)\left(1+ \frac{1}{kd}\right),
\end{align}
so that 
\begin{align}
\label{n2Est2}
k^{-1/2}\normt{\chi_L' \tilde{u}}{L^2(\R)} 
\leq C\left(1+\frac{1}{\sqrt{kL}}\right)\left(1+ \frac{1}{\sqrt{kd}}\right)\log{\left(2+\frac{1}{kd}\right)}.
\end{align}
\item 
For the term involving $\normt{\chi_L \tilde{u}'}{L^2(\R)}$ 
we first note that $|\tilde{u}'(s)|\leq (k/4)|H_1^{(1)}(kr(s))|$. The fact that $|H_1^{(1)}(z)|$ is decreasing implies that\footnote{This estimate is `worse than the worse case': it puts the source a distance $d/2$ above the midpoint of the set $F_{L\eps}$.}
\begin{align*}
\label{}
\normt{\chi_L \tilde{u}'}{L^2(\R)}^2 
\leq \int_{F_{L\eps}}|\tilde{u}'|^2
&\leq 2Ck^2\int_0^{L(1/2+\eps)} |H_1^{(1)}(kR(s))|^2 \,\rd s
=Ck I,
\end{align*}
where $R(s)=\sqrt{s^2+(d/2)^2}$ and $I=\int_0^l |H_1^{(1)}(\sqrt{z^2+\tilde{d}^2})|^2\,\rd z$ with $l=kL(1/2+\eps)$ and $\tilde{d}=kd/2$. To estimate $I$ we note that if $\tilde{d}>1$ then $\sqrt{z^2+\tilde{d}^2}\geq 1$ and so
\begin{align*}
\label{}
I\leq C\int_0^l \frac{\rd z}{\sqrt{z^2+1}}=C\sinh^{-1}{l} \leq C\log{(2+l)}.
\end{align*}
If $\tilde{d}\leq 1$ then we split the integral $I=I_1+I_2$ where $I_1=\int_0^{\min\{1,l\}}$ and $I_2=\int_{\min\{1,l\}}^l$. Then
\begin{align*}
\label{}
I_1\leq C\int_0^1 \frac{\rd z}{z^2+\tilde{d}^2} =\frac{C}{\tilde{d}}\tan^{-1}{\tilde{d}} \leq \frac{C}{\tilde{d}},
\end{align*}
and if $l>1$ then 
\begin{align*}
\label{}
I_2\leq C\int_1^l \frac{\rd z}{\sqrt{z^2+\tilde{d}^2}} \leq C\int_1^l \frac{\rd z}{z} = C\log{l}\leq C\log{2+l}.
\end{align*}
Collecting all these results gives that
\begin{align*}
\label{}
I\leq C\left(1+ \frac{1}{kd}\right) \log{\left(2+kL\right)},
\end{align*}
and hence that
\begin{align}
\label{n2Est3}
k^{1/2}\normt{\chi_L \tilde{u}'}{L^2(\R)} 
\leq C\left(1+ \frac{1}{\sqrt{kd}}\right)\log^{1/2}{\left(2+kL\right)}.
\end{align}
\end{itemize}

Finally, using \rf{n2Est1}, \rf{n2Est2} and \rf{n2Est3} in \rf{uChiEst} gives the required estimate \rf{eqn:H_one_half_Phi}.
When $n=3$ we have, in place of \rf{uChiEst},
\begin{align}
\label{uChiEst3D}
\normt{\chi_L \tilde{u}}{H^{1/2}_k(\R^2)} \leq k^{-1/2}\normt{\chi_L \tilde{u}}{H^{1}_k(\R^2)}
&\leq k^{1/2} \normt{\chi_L \tilde{u}}{L^2(\R^2)}
+  k^{-1/2} \sum_{j=1}^2\left(\norm{\pdone{\chi_L}{y_j}\tilde{u}}{L^2(\R^2)} + \norm{\chi_L \pdone{\tilde{u}}{y_j}}{L^2(\R)}\right).
\end{align}
Now, with $\by=(\tby,0)$ and $\tby=(y_1,y_2)\in R^2$, we have $\tilde{u}(\by)\equiv \tilde{u}(\tby)=\re^{\ri k r(\tby)}/(4\pi r(\tby))$, where $r(\tby)=\sqrt{(x_1-y_1)^2+(x_2-y_2)^2+x_3^2}$. Then $(\pdonetext{\tilde{u}}{y_j})(\tby)=(y_j-x_j)(\ri kr(\tby) -1)\re^{\ri k r(\tby)}/(4\pi r(\tby)^3)$. To avoid the singularity at $\by=\bx$ we define $d:=\dist(\bx, \Gamma)$ and take our closed set $F=B_\Gamma\setminus B_d$, where $B_\Gamma$ is the smallest closed ball containing $\Gamma$ and $B_d=\{\tby\in \R^2 : r(\tby)<d \}$. Again we set $\eps=\min\{1,d/(2L)\}$, and note that in this case ${\rm meas}(F_{L\eps})\leq \pi L^2(1/2+\eps)^2$ and ${\rm meas}(F_{L\eps}\setminus F)\leq \pi L \eps (L+2 d)$. \footnote{This holds because the definition of $\eps$ ensures that $L\eps \leq d$.} 
\begin{itemize}
\item
For the term involving $\normt{\chi_L \tilde{u}}{L^2(\R^2)}$, 
since $|\tilde{u}|\leq 1/(4\pi r(\tby))|$, we can estimate
\footnote{This estimate is `worse than the worse case': it puts the source a distance $d/2$ above the midpoint of the set $F_{L\eps}$.}
\begin{align*}
\label{}
\normt{\chi_L \tilde{u}}{L^2(\R^2)}^2 
&\leq C\int_0^{L(1/2+\eps)} \frac{r\rd r }{r^2 + (d/2)^2}\\
&= C \log{(1+(L(1+2\eps)/d)^2)}
\leq C \log{(2+L/d)}
\leq C \log{(2+1/(kd))} \log{(2+kL)},
\end{align*}
since $\eps\leq 1$. Hence
\begin{align}
\label{n3Est1}
k^{1/2}\normt{\chi_L \tilde{u}}{L^2(\R^2)}\leq Ck^{1/2} \log^{1/2}{(2+1/(kd))}\log^{1/2}{(2+kL)}.
\end{align}
\item 
For the term involving $\normt{(\pdonetext{\chi_L}{y_j})\tilde{u}}{L^2(\R^2)}$ we estimate
\begin{align*}
\label{}
\norm{\pdone{\chi_L}{y_j} \tilde{u}}{L^2(\R^2)}^2 
&= \int_{F_{L\eps}\setminus F} \left|\pdone{\chi_L}{y_j} \tilde{u}\right|^2
\leq \frac{C(L+2d)}{L\eps}\sup_{\tby\in F_{L\eps}\setminus F} \frac{1}{16\pi^2 r(\tby)^2}
\leq \frac{C(1+L/d)}{dL\eps}.
\end{align*}
Using \rf{FracEst} we obtain
\begin{align}
\label{n3Est2}
k^{-1/2}\norm{\pdone{\chi_L}{y_j} \tilde{u}}{L^2(\R^2)}
\leq Ck^{1/2}\left(1+\frac{1}{\sqrt{kL}}\right)\left(1+ \frac{1}{(kd)^{3/2}}\right)\left(1+\sqrt{kL}\right).
\end{align}
\item
For the term involving $\normt{\chi_L(\pdonetext{\tilde{u}}{y_j})}{L^2(\R^2)}$ 
we first note that $|(\pdonetext{\tilde{u}}{y_j})(\tby)|\leq (1+kr(\tby))/(4\pi r(\tby)^2)$. Then %
\footnote{This estimate is `worse than the worse case': it puts the source a distance $d/2$ above the midpoint of the set $F_{L\eps}$.}
\begin{align*}
\label{}
\norm{\chi_L \pdone{\tilde{u}}{y_j}}{L^2(\R)}^2 
\leq \int_{F_{L\eps}}\left|\pdone{\tilde{u}}{y_j}\right|^2
&\leq Ck^2\int_0^{l} \frac{\left(1+\sqrt{r^2+\tilde{d}^2}\right)^2 r \,\rd r}{(r^2+\tilde{d}^2)^2}\\
&\leq Ck^2\left(1+l+\tilde{d}\right)^2 \int_0^{l} \frac{r \,\rd r}{(r^2+\tilde{d}^2)^2}\\
&= Ck^2 \frac{\left(1+l+\tilde{d}\right)^2}{2\tilde{d}^2}\frac{l^2}{l^2+\tilde{d}^2}\\
&\leq Ck^2 \frac{\left(1+l+\tilde{d}\right)^2}{2\tilde{d}^2},
\end{align*}
where $l=kL(1/2+\eps)$ and $\tilde{d}=kd/2$. 
From this we find, recalling that $\eps\leq 1$, that
\begin{align}
\label{n3Est3}
k^{1/2}\norm{\chi_L \pdone{\tilde{u}}{y_j}}{L^2(\R^2)} 
\leq Ck^{1/2}\left(1+ \frac{1}{kd}\right)\left(1+ kL\right).
\end{align}
\end{itemize}

Finally, using \rf{n3Est1}, \rf{n3Est2} and \rf{n3Est3} in \rf{uChiEst3D} gives the required estimate \rf{eqn:H_one_half_Phi}.
\end{proof}

As an application we use Lemma \ref{lem:H_one_half_norms} to prove a $k$-explicit pointwise bound on the solution of the sound-soft screen scattering problem considered in Example \ref{ex:scattering}. 
\begin{cor}
\label{cor:BoundOnSolnInDomain}
The solution $u$ of problem $\sD''$, with $g_{\sD}=-u^i|_\Gamma$, satisfies the pointwise bound
\begin{align*}
\label{}
|u(\bx)|\leq 
\begin{cases}
 C\left(1+ \frac{1}{\sqrt{kL}}\right)\left(1+ \frac{1}{(kd)^{3/2}}\right)\left(1+(kL)^2\right), & n=3,\\
 C\left(1+\frac{1}{\sqrt{kL}}\right)\left(1+\frac{1}{\sqrt{kd}}\right)\log\left(2+\frac{1}{kd}\right)\log^{1/2}(2+kL)(1+\sqrt{kL}), & n=2,
\end{cases}
\end{align*}
where $\bx\in D$, $d:=\dist(\bx,\Gamma)$, $L:=\diam{\Gamma}$, and $C>0$ is independent of $k$ and $\Gamma$. 
\end{cor}
\begin{proof}
Using Theorem \ref{DirEquivThm} we can estimate
\begin{align*}
\label{}
|u(\bx)| 
= \left| \cS_k \left[\pdonetext{u}{\bn}\right](\bx)\right| 
&= \left|\langle \Phi_k(\bx,\cdot),\overline{\left[\pdonetext{u}{\bn}\right]} \rangle_{H^{1/2}(\Gamma)\times \tilde{H}^{-1/2}(\Gamma)}\right|\\
&\leq \norm{\Phi_k(\bx,\cdot)}{H^{1/2}(\Gamma)}\norm{S_k^{-1}}{H^{1/2}(\Gamma)\to\tilde{H}^{-1/2}(\Gamma)}\norm{u^i|_\Gamma}{H^{1/2}(\Gamma)},
\end{align*}
and the result follows from applying Lemma \ref{lem:H_one_half_norms} to estimate the first and third factors, and using Theorem \ref{ThmCoerc} (and the Lax Milgram lemma) to bound $\norm{S_k^{-1}}{H^{1/2}(\Gamma)\to\tilde{H}^{-1/2}(\Gamma)}\leq 2\sqrt{2}$.
\end{proof}

\section{Aperture problems}
\label{ScatProbHole}
In this section we show how the analysis of \S\ref{ScatProb}, where we studied scattering by bounded screens, can be modified to the complementary case of scattering by unbounded screens with bounded apertures. As in \S\ref{ScatProb} we let $\Gamma$ be a bounded and relatively open non-empty subset of $\Gamma_\infty:=\{ \bx\in\mathbb{R}^n:x_n=0\}$, but we now consider $\Gamma_\infty\setminus\overline{\Gamma}$ as the scatterer, with $\Gamma$ representing the aperture, so that the propagation domain is now $D':=U^+\cup U^-\cup \Gamma$.

With $W$ defined as in \S\ref{subsec:FuncSpacesTrace}, for $u\in W$ we define %
\begin{align*}
\label{}
\llbracket u \rrbracket &:=\gamma^+(\chi u)|_\Gamma-\gamma^-(\chi u)|_\Gamma\in H^{1/2}(\Gamma),\\
\llbracket \pdonetext{u}{\bn} \rrbracket &:=\partial^+_\bn(\chi u)|_\Gamma-\partial^-_\bn(\chi u)|_\Gamma\in H^{-1/2}(\Gamma),
\end{align*}
where $\chi$ is any element of $\scrD_{1,\Gamma}(\R^n)$. 
Further, for $u\in W$ satisfying
\begin{align}
\label{eqn:sumcond1}
\gamma^+(\chi u)|_{\Gamma_\infty\setminus\overline\Gamma} + \gamma^+(\chi u)|_{\Gamma_\infty\setminus\overline\Gamma} = 0, \quad \mbox{for all } \chi\in\scrD(\R^n),
\end{align}
we define
\begin{align*}
\label{}
\sumpm{u}&:=\gamma^+(\chi u)+\gamma^-(\chi u)\in H^{1/2}_{\overline{\Gamma}}
\end{align*}
where $\chi$ is any element of $\scrD_{1,\Gamma}(\R^n)$. For $u\in W$ satisfying
\begin{align}
\label{eqn:sumcond2}
\partial_\bn^+(\chi u)|_{\Gamma_\infty\setminus\overline\Gamma}+ \partial_\bn^+(\chi u)|_{\Gamma_\infty\setminus\overline\Gamma} = 0, \quad \mbox{for all } \chi\in\scrD(\R^n),
\end{align}
we define
\begin{align*}
\label{}
\sumpm{\pdonetext{u}{\bn}}&:=\partial^+_\bn(\chi u)+\partial^-_\bn(\chi u)\in H^{-1/2}_{\overline{\Gamma}},
\end{align*}
where $\chi$ is any element of $\scrD_{1,\Gamma}(\R^n)$. 

We then have the following version of Green's representation theorem for an aperture:
\begin{thm}%
\label{GreenRepThmAperture}
Let $u\in C^2(U^\pm)\cap W$ satisfy \rf{eqn:sumcond1} and \rf{eqn:sumcond2} with $(\Delta + k^2)u=0$ in $U^\pm$ and suppose that $u$ satisfies the Sommerfeld radiation condition at infinity in $U^\pm$. Then
\begin{align}
\label{RepThmGeneralAperture}
u(\bx) =\mp\cS_k\sumpm{\pdonetext{u}{\bn}}(\bx) \pm \cD_k\sumpm{u}(\bx), \qquad\bx\in U^\pm.
\end{align}
\end{thm}
\begin{proof}
Follows the proof of Theorem \ref{GreenRepThm}, except one takes the \emph{difference} rather than the sum of the two representation formulas obtained in the domains $U^\pm_R$. That the contributions from $\Gamma_\infty\cap B_R(0)\setminus\overline{\Gamma}$ cancel follows from the assumption that \rf{eqn:sumcond1} and \rf{eqn:sumcond2} hold.
\end{proof}

We will consider the following two BVPs, formulated by imposing transmission conditions across the aperture $\Gamma$.
\begin{defn}[Problem $\sH$]
Given $g_{\sH}\in H^{-1/2}(\Gamma)$, find $u\in C^2\left(U^\pm\right)\cap W$ such that
\begin{align}
 \Delta u+k^2u & =  0, \quad \mbox{in }U^\pm, \label{eqn:HE3} \\
  u &= 0, \quad \mbox{on }\Gamma_\infty\setminus\overline{\Gamma}, \label{eqn:bc3a}\\
   \llbracket \pdonetext{u}{\bn} \rrbracket  & = g_{\sH}, \label{eqn:bc3b}\\
     [u] & = 0, \label{eqn:bc3c}\\
 \sumpm{u} &\in \tilde H^{1/2}(\Gamma), \label{eqn:bc3d}
\end{align}
and $u$ satisfies the Sommerfeld radiation condition.
\end{defn}
By \rf{eqn:bc3a} we mean, precisely, that $\gamma^\pm (\chi u)\vert_{\Gamma_\infty\setminus\overline{\Gamma}}=0$, where $\chi$ is any element of
$\scrD(\R^n)$.
\begin{defn}[Problem $\sI$]
Given $g_{\sI}\in H^{1/2}(\Gamma)$, find $u\in C^2\left(U^\pm\right)\cap W$ such that
\begin{align}
 \Delta u+k^2u & =  0, \quad \mbox{in }U^\pm, \label{eqn:HE4} \\
  \pdone{u}{\bn}  &= 0, \quad \mbox{on }\Gamma_\infty\setminus\overline{\Gamma}, \label{eqn:bc4a}\\
      \llbracket u \rrbracket  & = g_{\sI}, \label{eqn:bc4b}\\
[\pdonetext{u}{\bn} ] & = 0, \label{eqn:bc4c}\\
 \sumpm{\pdonetext{u}{\bn}} &\in \tilde H^{-1/2}(\Gamma), \label{eqn:bc4d}
\end{align}
and $u$ satisfies the Sommerfeld radiation condition.
\end{defn}
By \rf{eqn:bc4a} we mean, precisely, that $\partial_\bn^\pm (\chi u)\vert_{\Gamma_\infty\setminus\overline{\Gamma}}=0$, where $\chi$ is any element of
$\scrD(\R^n)$.

\begin{example}
Consider the aperture problem of scattering by $\Gamma_\infty\setminus\overline{\Gamma}$ of a plane wave
\begin{align}
\label{uiDefAp}
  u^i(\bx)&:=\re^{\ri  k \bx\cdot \bd}, \qquad \bx\in U^+,%
\end{align}
incident from $U^+$, where
$\bd=(\tbd,d_n)\in\mathbb{R}^n$ is a unit direction vector with $\tbd\in\R^{n-1}$ and $d_n<0$. The cases of a `sound soft' and a `sound hard' screen are modelled respectively by problem $\sH$ (with $g_{\sH}=-2\pdonetext{u^i}{\bn}|_\Gamma$) and problem $\sI$ (with $g_{\sI}=-2u^i|_\Gamma$). In both cases $u$ represents the diffracted field, the total field being given by
\begin{align}
\label{}
u^{\rm tot} =
\begin{cases}
u + u^i + u^r, & \mbox{in }U^+,\\
u , & \mbox{in }U^-,\\
\end{cases}
\end{align}
where the reflected wave $u^r(\bx):=c\re^{\ri  k \bx\cdot \bd'}$, where $\bd'=(\tbd,-d_n)$ and $c=-1$ for the sound soft case and $c=1$ for the sound hard case.
\end{example}

Using Theorems \ref{LayerPotRegThm} and \ref{GreenRepThmAperture}, it is straightforward to show that the BVPs $\sH$ and $\sI$ are equivalent to the BIEs \rf{BIE_hyp} and \rf{BIE_sl} respectively. Conditions \rf{eqn:bc3d} and \rf{eqn:bc4d} combined with Theorems \ref{ThmSk} and \ref{ThmTk} then imply that problems $\sH$ and $\sI$ are uniquely solvable. Although we do not provide full details here, we note that the conditions \rf{eqn:bc3c} and \rf{eqn:bc4c} in problems $\sH$ and $\sI$ are required to ensure that \rf{eqn:sumcond1} and \rf{eqn:sumcond2} hold for both problems, so that the representation theorem is valid. Specifically, if $u$ is a solution of problem $\sH$ then set $w(\bx):=u(\bx)-u(\bx')$, where $\bx'$ denotes the reflection of $\bx$ in $\Gamma_\infty$. From \rf{eqn:bc3a} and \rf{eqn:bc3c} and the uniqueness of the solution of the Helmholtz equation in a half-space with Dirichlet boundary conditions it follows that $w=0$, i.e.\ that $u(\bx)=u(\bx')$, which implies \rf{eqn:sumcond2}. Similarly, one can prove that if $u$ is a solution of problem $\sI$ then $u(\bx)=-u(\bx')$, which implies \rf{eqn:sumcond1}. These results are summarised in the following theorems.
\begin{thm}
\label{DirHoleExUn}
For any $g_{\sH}\in H^{-1/2}(\Gamma)$ problem $\sH$ has a unique solution given by formula 
\begin{align*}
\label{}
 u(\bx ) = \pm \cD_k\sumpm{u}(\bx), \qquad\bx\in U^\pm,
\end{align*}
where $\sumpm{u}$ is the unique solution in $\tilde H^{1/2}(\Gamma)$ of equation \rf{BIE_hyp} with $g_{\sN}=g_{\sH}/2$.
\end{thm}
\begin{thm}
\label{NeuHoleExUn}
For any $g_{\sI}\in H^{1/2}(\Gamma)$ problem $\sI$ has a unique solution given by formula 
\begin{align*}
\label{}
u(\bx )=\mp \cS_k\sumpm{\pdonetext{u}{\bn}}(\bx), \qquad\bx\in U^\pm,
\end{align*}
where $\sumpm{\pdonetext{u}{\bn}}$ is the unique solution in $\tilde H^{-1/2}(\Gamma)$ of equation \rf{BIE_sl} with $g_\sD=g_\sI/2$.
\end{thm}
Note in particular that while the scattering problem for a bounded sound-soft screen requires the solution of the single-layer BIE \rf{BIE_sl} on the screen, the problem for an aperture in a sound-soft screen requires the solution of the hypersingular BIE \rf{BIE_hyp} on the aperture.

\section{Acknowledgements}
The authors are grateful to A.\ Moiola for many stimulating discussions in relation to this work.
\bibliography{BEMbib_short}%
\bibliographystyle{siam}
\end{document}